\documentclass[11pt,reqno,a4paper]{amsart} 
\usepackage{graphicx}
\usepackage{amssymb,amsmath}
\usepackage{amsthm}
\usepackage{color,graphicx}
\usepackage{hyperref}
\usepackage{color}
\usepackage{mathabx}
\usepackage{epstopdf}
\usepackage{xcolor}
\usepackage[normalem]{ulem}

\usepackage{subfigure} 

\usepackage{verbatim}

\newcommand{\be}{\begin{equation}}

\newcommand{\ee}{\end{equation}}

\newcommand{\cn}{{\rm \,cn}}
\newcommand{\sn}{{\rm \,sn}}
\newcommand{\dn}{{\rm \,dn}}

\newcommand{\R}{{\mathbb R}}

\newcommand{\Ker}{{\rm \,Ker}}
\newcommand{\K}{{\rm \,K}}
\newcommand{\E}{{\rm \,E}}

\numberwithin{equation}{section}
\numberwithin{figure}{section}

\newtheorem{theorem}{Theorem}[section]
\newtheorem{proposition}[theorem]{Proposition}
\newtheorem{remark}[theorem]{Remark}
\newtheorem{lemma}[theorem]{Lemma}
\newtheorem{corollary}[theorem]{Corollary}
\newtheorem{definition}[theorem]{Definition}

\begin{document}
\vglue-1cm \hskip1cm
\title[Cnoidal Waves for the KG and NLS]{Cnoidal Waves for the cubic nonlinear Klein-Gordon and Schr\"odinger Equations
}

\begin{center}

\subjclass[2000]{35Q51, 35Q55, 35Q70.}

\keywords{cubic Klein-Gordon equation, cubic Schr\"odinger equation, cnoidal waves, orbital instability, orbital stability.}

\maketitle

{\bf Guilherme de Loreno}

{Departamento de Matem\'atica - Universidade Estadual do Centro Oeste, Campus Cedeteg\\
	Alameda \'Elio Antonio Dalla Vecchia, 838, CEP 85040-167, Guarapuava, PR, Brazil.}\\
{ guilhermeloreno@unicentro.br}

{\bf Gabriel E. Bittencourt Moraes}

{Departamento de Matem\'atica - Universidade Estadual de Londrina\\
	Rodovia Celso Garcia Cid, PR-445, Km 380, CEP 86057-970, Londrina, PR, Brazil.}\\
{ ge.bitt.moraes@gmail.com}

{\bf F\'abio Natali}

{Departamento de Matem\'atica - Universidade Estadual de Maring\'a\\
	Avenida Colombo, 5790, CEP 87020-900, Maring\'a, PR, Brazil.}\\
{ fmanatali@uem.br}

{\bf Ademir Pastor}

Departamento do Matemática, Instituto de Matemática, Estatística e Computação Científica (IMECC), Universidade Estadual de Campinas (UNICAMP), \\
Rua S\'ergio Buarque de Holanda, 651,  13083-859, Campinas--SP, Brazil\\
{ apastor@ime.unicamp.br}

\vspace{3mm}

\end{center}

\begin{abstract}
In this paper, we establish orbital stability results for \textit{cnoidal} periodic waves of the cubic nonlinear Klein-Gordon and Schrödinger equations in the energy space restricted to zero mean periodic functions.  More precisely, for one hand, we prove that the cnoidal waves of the cubic Klein-Gordon equation are orbitally unstable as a direct application of the theory developed by Grillakis, Shatah, and Strauss. On the other hand, we show that the cnoidal waves for the Schrödinger equation are orbitally stable by constructing a suitable Lyapunov functional restricted to the associated zero mean energy space.  The spectral analysis of the corresponding linearized operators, restricted to the periodic Sobolev space consisting of zero mean periodic functions, is performed using the Floquet theory and a Morse Index Theorem.

\end{abstract}

\section{Introduction} 

This paper  presents new orbital stability results for the cubic Klein-Gordon equation (KG)
\begin{equation}\label{KF2}
u_{tt}-u_{xx}+u-u^{3}=0,
\end{equation} and the cubic nonlinear Schr\"odinger equation (NLS)
\begin{equation}\label{KF1}
i u_t+u_{xx}+|u|^2 u=0.
\end{equation}
In both equations, $u:\mathbb{R} \times \mathbb{R}_+ \rightarrow \mathbb{B}$ is an $L$-periodic function in the spatial variable, $\mathbb{B}=\mathbb{R}$ for the KG equation and $\mathbb{B}=\mathbb{C}$ for the NLS equation. The KG equation was first proposed as a relativistic
generalization of the Schr\"odinger equation describing
free particles; however, it also has several applications in physics and
engineering including quantum field theory,  dispersive wave phenomena, and vibrating
systems in classical mechanics (see, for instance, \cite{gravel}). Specifically, the cubic nonlinearity has been used as a model equation in field theory (see \cite{dashen}). The NLS equation also appears in many applications in physics and engineering, such as nonlinear optics, quantum mechanics, and nonlinear waves (see, for instance, \cite{boyd} and \cite{fibich}).

Let us first pay particular attention to \eqref{KF2}.
It is well known that \eqref{KF2} can be seen as an abstract Hamiltonian system. In fact, by setting
\begin{equation}\label{J}
J=\left(
\begin{array}{cc}
 0 & 1\\
-1 & 0\end{array}
\right),
\end{equation}
we see that $(\ref{KF2})$ writes as
\begin{equation}\label{hamilt31}
	\frac{d}{dt}U(t)=J E'(U),
\end{equation} 
where $U=(u,u_t)=:(u,v)$ and $E'$ indicates the Fr\'echet derivative of the conserved quantity 
\begin{equation}\label{E}
E(u,v)
=\displaystyle\frac{1}{2}\int_{0}^{L} \left[u_x^2+v^2+u^2\left(1-\frac{u^2}{2}\right)\right] \;dx.
\end{equation}
 Moreover, \eqref{KF2} also conserves the quantity  
\begin{equation}\label{F}
	F(u,v)=\displaystyle\int_{0}^{L}u_x v \; dx.
\end{equation}

An important mathematical aspect concerning equation \eqref{KF2} is the existence of periodic traveling wave solutions  of the form
\begin{equation}\label{PW}
u(x,t)=\varphi(x-ct),
\end{equation}
where $c \in \mathbb{R}$ represents the wave speed and $\varphi=\varphi_c: \mathbb{R}\rightarrow \mathbb{R}$ is an $L$-periodic smooth function, meaning that $\varphi(\xi+L)=\varphi(\xi)$ for all $\xi\in\mathbb{R}$. Existence and stability of periodic waves for the KG equations (in its complex form) have appeared in   \cite{NP2008}, where the authors determined the orbital stability of \textit{dnoidal} standing wave solutions of the form $u(x,t)=e^{i\omega t}\varphi(x)$, where $\omega\in\mathbb{R}$ is the so called frequency of the wave, and orbital instability of \textit{cnoidal} standing wave solutions. The main tool to obtain these results was the classical Grillakis, Shatah and Strauss theory in \cite{grillakis1} and \cite{grillakis2} in the periodic context. We highlight that the existence of standing waves and the spectral analysis of the corresponding linearized operator in the case of dnoidal and cnoidal solutions for the cubic complex KG equation is quite similar to that determined in \cite{angulo} for the cubic NLS where dnoidal and cnoidal standing waves have also been considered. In \cite{cardoso2013} the authors considered the quintic KG equation and showed the orbital instability of periodic \textit{dnoidal} waves restricted to perturbations in the subspace of the even periodic functions in $H_{per}^1$. The main tool used was a computational approach to obtain the behaviour of the non-positive spectrum of the linearized operator combined with the theory in \cite{grillakis1}. Linear and spectral stability of periodic waves to the KG equation with general power-like nonlinearity have appeared in \cite{dem} and \cite{hakkaev}. In all these works the perturbations were considered with the same period of the underlying wave.

 In \cite{bronski}, the authors proved spectral stability results for a general second order PDE using the quadratic pencils technique in order to obtain a precise counting for the Hamiltonian Krein Index $\mathcal{K}_{\rm Ham}$, an important tool which provides us with indications for determining the spectral stability and instability for abstracts Hamiltonian systems. In fact, if $\mathcal{K}_{\rm Ham}=0$, the periodic wave is spectrally stable and orbital stability results can be determined if the associated Cauchy problem has global solutions in time for arbitrary initial data in the energy space. Now, if $\mathcal{K}_{\rm Ham}$ is an odd positive number, the periodic wave is spectrally unstable, while if $\mathcal{K}_{\rm Ham}$ is an even positive number, no definitive conclusion can be drawn about the spectral stability of the periodic wave. However, even though our cnoidal periodic wave in $(\ref{Sol2})$ could be spectrally stable according to \cite{bronski}, we can not conclude an orbital stability result in the energy space. In fact,  the model $(\ref{KF2})$ does not have global solutions in time for arbitrary initial data $(u_0,v_0)\in H_{per}^1\times L_{per}^2$ since the corresponding evolution $(u(t),v(t))\in H_{per}^1\times L_{per}^2$ may blow up in finite time (see \cite{levine}).

 Consider the following Klein-Gordon type equation
\begin{equation}u_{tt}-u_{xx}+V'(u)=0.\label{genKG}\end{equation}
Important results concerning spectral/modulational stability of periodic waves have been determined in \cite{jones1} and \cite{jones2} for the case where $V:\mathbb{R}\rightarrow \mathbb{R}$ is a  periodic nonlinearity (both works include the case $V(u)=-\cos(u)$; the well known sine-Gordon equation). In a similar setting of assumptions as in \cite{jones1}, the authors in \cite{MM} gave a simple criteria for the existence of dynamical Hamiltonian-Hopf instabilities, an useful tool to establish the spectral stability of periodic traveling waves. Some additional references of related topics can be listed as \cite{hakkaev1} and \cite{stan}.

\indent Here we are interested in the case where the periodic function $\varphi$ is a sign changing function. Substituting \eqref{PW} into \eqref{KF2}, we have that $\varphi$ satisfies the following ODE
\begin{equation}\label{KF3}
-(1-c^2) \varphi''+\varphi-\varphi^{3}=0.
\end{equation}
For $1-c^2>0$, an explicit solution of \eqref{KF3}, depending on the Jacobi elliptic function of cnoidal type is
\begin{equation}\label{Sol2}
\varphi(x)=\frac{\sqrt{2}k}{\sqrt{2k^2-1}}\;{\rm cn}\left(\frac{4\K(k)}{L}x,k \right),\; \; x \in \mathbb{R},
\end{equation} 
where $k\in\left(\frac{1}{\sqrt{2}},1\right)$ is the modulus of the elliptic function and $\K(k)$ is the complete  elliptic integral of the first kind. The parameter $\omega:=1-c^2$ then depends smoothly on $k\in\left(\frac{1}{\sqrt{2}},1\right)$ and may be written as
\begin{equation}\label{ccn}
\omega =  \frac{1}{16}\frac{L^2}{\K(k)^2 (2k^2-1)}.
\end{equation}
Since $0<\omega<1$, from \eqref{ccn}, we formally obtain $0<L<4\K(k)\sqrt{2k^2-1}$ which implies that the period  may be chosen in the interval $(0,+\infty)$, as $k$ varies in the interval $\left(\frac{1}{\sqrt{2}},1\right)$  (see additional details in Proposition $\ref{cnoidalsol}$). The cnoidal wave in $(\ref{Sol2})$ (which is similar to $(\ref{Sol3})$) is a typical periodic wave with the zero mean property. Recent results concerning existence and stability of periodic waves satisfying the zero mean condition have been determined in \cite{ABS}, \cite{ACN}, \cite{NPU}, \cite{NPU1}, and references therein.

Before proceeding, since \eqref{KF2} is invariant by translations, let us recall the general stability/instability criterion for Hamiltonian systems established in \cite{grillakis1}.  By defining
$G(u,v)=E(u,v)-cF(u,v),$ we see that any solution of \eqref{KF3} satisfies $G'(\varphi,c\varphi')=0$, that is, $(\varphi,c\varphi')$ is a critical point of $G$.
Assume the following three conditions:
\begin{itemize}
	\item[(i)] there exists a smooth curve of solution for \eqref{KF3}, $c \in I:= (-1,1)  \mapsto\varphi_c \in H_{per,m}^2,$ where each $\varphi:=\varphi_c$ has period $L$;
	\item[(ii)] the linearized operator  \begin{equation}\displaystyle\label{opconstrained2}\mathcal{L}_{\Pi}:= G''(\varphi,c\varphi')+\displaystyle \left(
\displaystyle	\begin{array}{ccc}
		\frac{3}{L}\int_{0}^{L} \varphi^2 \cdot \; dx& & 0 \\\\
		0  & & 0
	\end{array}\right)=\displaystyle\mathcal{L}+\displaystyle \left(
\begin{array}{ccc}
\frac{3}{L}\int_{0}^{L} \varphi^2 \cdot \; dx & & 0 \\\\
0  & & 0
\end{array}\right),\end{equation} where $\mathcal{L}$ is given by
	\begin{equation}\label{matrixop313}
		\displaystyle \mathcal{L}=\left(
		\begin{array}{ccc}
			-\partial_x^2-3\varphi^2+1  & &c\partial_x\\\\
			\ \ \ \ -c\partial_x & & 1
		\end{array}\right),
	\end{equation}
	has only one negative eigenvalue which is simple and zero is a simple eigenvalue associated to the eigenfunction $(\varphi',c\varphi'')$;
	\item[(iii)] the function $\mathsf{d}: I \rightarrow \mathbb{R}$ defined by $\mathsf{d}(c)=G(\varphi,c\varphi')$ is non-degenerated, that is, $\mathsf{d}''(c)\neq0$.
\end{itemize}
Then according to \cite{grillakis1},  if $\mathsf{d}''(c)<0$ we have the orbital instability and if $\mathsf{d}''(c)>0$ we have the orbital stability (see Definition \ref{stadef} for the precise definition of stability). Since  $G'(\varphi,c\varphi')=0$, we immediately deduce that $\mathsf{d}'(c)=-F(\varphi,c\varphi')$, which implies, from the chain rule, that
\begin{equation}\label{dsegundaa}
	\begin{split}
	\mathsf{d}''(c)&= -\int_0^{L}(\varphi'(x))^2dx-c\dfrac{d}{dc}\int_0^{L}(\varphi'(x))^2dx\\
	&=-\int_0^{L}(\varphi'(x))^2dx+2(1-\omega)\frac{d}{d\omega}\int_0^{L}(\varphi'(x))^2dx,
	\end{split}
\end{equation} 
where we used that $\omega=1-c^2$.

Next, we shall briefly explain how to obtain (ii) and (iii) in our case. Associated with  $(\ref{matrixop313})$ is the auxiliary operator (see Section \ref{section4}) 
\begin{equation}\label{linoperat}
	\mathcal{L}_1=- \omega \partial_x^2-3\varphi^2+1, 
\end{equation}
which has exactly two negative simple eigenvalues and zero is a simple eigenvalue with eigenfunction $\varphi'$. The spectral properties of the operator $\mathcal{L}_1$ are based on basic facts from the Floquet theory, which can be found in \cite{eas} and \cite{magnus}. As far as we know, these facts do not allow (using, for instance, the min-max characterization of eigenvalues) neither to decide about the simplicity of the zero eigenvalue  nor to calculate the quantity of negative eigenvalues of $\mathcal{L}$. So, at a first glance, we are not able to obtain the stability or the instability of the solution $\varphi$. To overcome this difficulty, instead of $\mathcal{L}_1$ we consider the constrained operator $\mathcal{L}_{{1\Pi}}: D(\mathcal{L}_{{1\Pi}})=H^2_{per} \cap L^2_{per,m}\subset  L^2_{per,m} \rightarrow L^2_{per,m}$ defined by
\begin{equation}\label{opconstrained12}
	\mathcal{L}_{{1\Pi}} \, f:= \mathcal{L}_1 f+\frac{3}{L}\int_0^L \varphi^2 f \; dx.
\end{equation}
 Here, $L^2_{per,m}$ indicates the closed subspace of $L^2_{per}$ constituted by the zero mean  functions. In addition, we observe that $\int_{0}^L\mathcal{L}_{1\Pi}fdx=0$ for all $f\in D(\mathcal{L}_{{1\Pi}})$, which means that $\mathcal{L}_{1\Pi}$ is well-defined in $L_{per,m}^2$. Associated to $\mathcal{L}_{1\Pi}$ is the constrained space $S_1:= [1] \subset \Ker(\mathcal{L}_1)^{\perp}=[\varphi']^{\bot}$ and the number
$
D_1:=(\mathcal{L}_1^{-1} 1,1)_{L^2_{per}}.
$
A Morse index theorem is a useful tool for obtaining spectral properties of linear operators. At this point, we are in a position to apply an index theorem for self-adjoint operators (see \cite[Theorem 5.3.2]{kapitula} and \cite[Theorem 4.1]{pel-book}), which provides  the spectral properties of $\mathcal{L}_{1\Pi}$ in terms of the spectral properties associated with $\mathcal{L}_1$. More precisely, since $\ker(\mathcal{L}_1)=[\varphi']$, we have 
\begin{equation}\label{indexformula12}
	\text{n}(\mathcal{L}_{{1\Pi}})=\text{n}(\mathcal{L}_1)-n_0-z_0
\end{equation}
and
\begin{equation}\label{indexformula123}
	\text{z}(\mathcal{L}_{{1\Pi}})=\text{z}(\mathcal{L}_1)+z_0,
\end{equation}
where $\text{n}(\mathcal{A})$ and $\text{z}(\mathcal{A})$ indicate the number of negative eigenvalues and the dimension of the kernel of the linear operator $\mathcal{A}$. In addition, numbers $n_0$ and $z_0$ are defined respectively as
\begin{equation}\label{n0z0}
n_0:=
\begin{cases}
	1, \: \text{if} \: D_1<0, \\
	0, \: \text{if} \: D_1 \geq 0,\ \
\end{cases}
\quad \text{and} \quad\ 
z_0:=
\begin{cases}
	1, \: \text{if} \: D_1=0, \\
	0, \: \text{if} \: D_1 \neq 0.
\end{cases}
\end{equation}
\indent Let $L>0$ be fixed. In our approach, we consider the wave speed $c$ only over the interval $c\in [0,1)$. The case $-1<c<0$ can be analyzed similarly by symmetry. We shall prove that if $c\in [0,1)$ then there exists a unique $c^{*}>0$ such that for any $c\in [0,c^*)$ we have $D_1<0$. In view of $(\ref{indexformula12})$ and \eqref{indexformula123} we deduce that $\mathcal{L}_{1\Pi}$ has only one negative eigenvalue which is simple and zero is a simple eigenvalue whose eigenfunction is $\varphi'$. These facts imply, by using the index formula in $(\ref{indexformula12})$ for $\mathcal{L}_{\Pi}$ instead of $\mathcal{L}_{1\Pi}$ and $\mathcal{L}$ instead of $\mathcal{L}_1$, that the number of negative eigenvalues of the constrained operator $\mathcal{L}_{\Pi}$ is equal to one and zero is a simple eigenvalue whose eigenfunction is $(\varphi',c\varphi'')$.\\
\indent Also, by using the explicit expression of the cnoidal waves in $(\ref{Sol2})$ and some algebraic computations we show that $\mathsf{d}''(c)<0$, $c\in [0,c^*)$. Consequently, we may apply the abstract theory in \cite{grillakis1} as described above to conclude the orbital instability of the cnoidal waves for $c$ in a smaller subset, namely for $c\in [0,c(k_1))\subset [0,c^{\ast})$, where $k_1\approx 0.802$.

Next, we turn attention to the NLS equation \eqref{KF1}. As with the KG equation, the NLS equation can also be viewed as a (real) Hamiltonian system. Indeed, by setting   $u = P + iQ$, where $P$ and $Q$ are, respectively, the real and imaginary parts of $u$  and writing $U =(P,Q)$, we see that (\ref{KF1}) reads as
\begin{equation}\label{hamiltonian-schrodinger}
	\dfrac{d}{dt} U(t) = J\mathcal{E}'(U(t)), 
\end{equation}
where $J$ is as in \eqref{J} and $\mathcal{E}'$ represents the Fr\'echet derivative of the conserved quantity
\begin{equation}\label{quantEF1}
	\mathcal{E}(U):=\mathcal{E}(u)= \frac{1}{2}\int_0^L\left[ |u_x|^2-\frac{1}{2}|u|^4\right]\; dx.
\end{equation}
It is well-known that (\ref{KF1}) also conserves the quantity 
\begin{equation}\label{quantEF2}
	\mathcal{F}(u)=\frac{1}{2}\int_0^L |u|^2\; dx.
\end{equation}

Here we look for the existence of standing waves  of the form
\begin{equation}\label{PW2}
	u(x,t)=e^{i\omega t} \varphi_{\omega}(x), 
\end{equation}
where $\omega \in \mathbb{R}$ and $\varphi:=\varphi_\omega : \mathbb{R} \rightarrow \mathbb{R}$ is a smooth and $L$-periodic function.  For the NLS equation, by using the ideas introduced  in \cite{bona}  and \cite{weinstein}, the author in \cite{angulo} established stability properties of periodic standing waves solutions with \textit{dnoidal} profile with respect to perturbations with the same period $L$ (see also \cite{gallay1} and \cite{lecoz}). Existence of smooth branches of solutions with \textit{cnoidal} profiles (see \eqref{Sol3}) were also reported in \cite{angulo}; however, the author was not able to obtain the  orbital stability/instability in the energy space $ H^1_ {per}$ for these waves. By using the techniques introduced in \cite{grillakis1} and \cite{grillakis2}, the cnoidal waves were shown to be orbitally stable  in \cite{gallay1} and \cite{gallay2} with respect to anti-periodic perturbations. Spectral stability with respect to bounded or localized perturbations were also reported in \cite{gallay1}. For $\omega$ in some interval $ (0,\omega^*)$ (see Proposition \ref{propn1} below), the authors in \cite{lecoz} have established spectral stability results for the cnoidal waves with respect to perturbations with the same period  $L$ and orbital stability results in the space  constituted by anti-periodic functions with period $L/2$. Their proofs relies on first  proving that the cnoidal waves satisfy a convenient minimization problem with constraints, which yields the orbital stability. The spectral stability follows by relating the coercivity of the linearized action  with the number of eigenvalues with negative Krein signature of $J\mathcal{L}$ (see \eqref{matrixop2}).

 The integrability of the equation \eqref{KF1} was used by the authors in \cite{deconinckupsal} to determine spectral stability results  of periodic waves (including dnoidal and cnoidal waves) with respect to subharmonic perturbations (e.g. perturbation of integer multiple $n$ times the minimal period of the solution). The spectral stability is then used to conclude the orbital stability for cnoidal waves as in $(\ref{Sol3})$. However, the authors employed the arguments in \cite{grillakis1} to conclude the orbital stability by considering the orbit generated only by one symmetry of equation \eqref{KF1}. A complementary result with similar ideas as in \cite{deconinckupsal} was established in \cite{BND} for the equation $(\ref{KF1})$ with defocusing nonlinearity.\\
\indent To the best of our knowledge, the orbital stability (with the orbit contemplating the two basic symmetries, namely, rotation and translation) of the cnoidal waves in the energy space with respect to perturbations with the same period of the underlying wave remains as an open problem. Since the cnoidal waves has zero mean on its fundamental period,  here we advance in this question by showing the orbital stability with respect to \textit{perturbations with the same period and with the zero mean property}.

	Before describing our results we point out that strong solutions (that is, those ones in $H_{per}^2$) of the NLS equation do not preserve their mean as in the case of the Korteweg-de Vries equation (see \cite{ABS}). In particular, if the initial data has zero mean then not necessarily the strong solutions of \eqref{KF2} or \eqref{KF1} have zero mean. However, as we will see below, we can always prove the existence of global \textit{weak solutions} in the space of zero mean functions. Hence, we emphasize that along the paper all solutions of the Cauchy problem associated with \eqref{KF2} or \eqref{KF1} must be understood in the weak sense. As far as we can see, this fact is consistent with the abstract theory of Grillakis, Shatah, and Strauss where only the existence of weak solutions is assumed (see \cite[page 165]{grillakis1}).

In order to precisely describe our result, observe that replacing \eqref{PW2} into \eqref{KF1}, we obtain the ODE
\begin{equation}\label{KF4}
	- \varphi''+\omega \varphi-\varphi^{3}=0,
\end{equation}
which is quite similar to \eqref{KF3}. For $\omega>0$, we can find explicit solutions depending on the Jacobi elliptic function of cnoidal type as
\begin{equation}\label{Sol3}
	\varphi(x)=\frac{\sqrt{2\omega}k}{\sqrt{2k^2-1}}\;{\rm cn}\left(\frac{4\K(k)}{L}x,k \right),
\end{equation} 
where $k \in \left(\frac{1}{\sqrt{2}}, 1\right)$. The frequency of the wave  is explicitly given  by
\begin{equation}\label{omegacn2}
	\omega = \frac{16 \K(k)^2 (2k^2-1)}{L^2}.
\end{equation}
By $(\ref{omegacn2})$ we can choose a fixed period $L\in (0,+\infty)$ in order to construct smooth periodic waves depending on $\omega>0$ (see Proposition $\ref{cnoidalsol2}$). 

Our approach to obtain the stability result will be based on the construction of a suitable Lyapunov function combined with the convexity of the function  $\mathsf{d}(\omega)=\mathcal{E}(\Phi)+\omega\mathcal{F}(\Phi)$, where $\Phi=(\varphi,0)$. It is not difficult to see that $\mathsf{d}''(\omega)>0$ (see \cite[page 23]{angulo}). In order to construct the Lyapunov function we need to study the spectral properties of the linearized operator
\begin{equation}\label{matrixop2}
	\mathcal{L} :=\mathcal{G}''(\Phi)= \displaystyle \left(
	\begin{array}{ccc}
	\mathcal{L}_2  & &  0 \\\\
	 0 & & \mathcal{L}_3
	\end{array}\right),
\end{equation}
where $\mathcal{G}(u)=\mathcal{E}(u)+\omega\mathcal{F}(u)$ and 
\begin{equation}\label{L2L3}
	\mathcal{L}_2: = -\partial_x^2+\omega-3\varphi^2 \quad \text{and} \quad \mathcal{L}_3:=  -\partial_x^2+\omega-\varphi^2.
\end{equation}
The operators $\mathcal{L}_2$ and $\mathcal{L}_3$ are such that  $\text{n}(\mathcal{L}_2)=2$, $\text{n}(\mathcal{L}_3)=1$ and $\text{z}(\mathcal{L}_2)=\text{z}(\mathcal{L}_3)=1$; in addition $\ker(\mathcal{L}_2)=[\varphi']$ and $\ker(\mathcal{L}_3)=[\varphi]$ (see \cite[Theorems 3.2 and 3.4]{angulo}). Since $\mathcal{L}$ is a diagonal operator, this allow us to conclude that $\text{n}(\mathcal{L})=3$ and $\ker(\mathcal{L})=[(\varphi',0),(0,\varphi)]$. At this point, it must be highlighted that, as we already commented above, the abstract theories for Hamiltonian systems developed in \cite{grillakis1} and \cite{grillakis2} do not give a positive answer concerning the stability of the cnoidal waves with respect to perturbations with the same period because 
$
	\text{n}(\mathcal{L})-p(\mathsf{d}'')=3-1=2
$
is an even number. Recall that here, $p(\mathcal{A})$ indicates the number of positive eigenvalues of the  linear operator $\mathcal{A}$. One way to overcome this difficulty is to restrict attention to the subspace $L^2_{per,m}$. Indeed, let us consider the constrained operators
\begin{equation}\label{opconstrained123}
	\mathcal{L}_{{2\Pi}} \, f:= \mathcal{L}_2 f+\frac{3}{L}\int_0^L \varphi^2 f \; dx
\end{equation}
and
\begin{equation}\label{opconstrained1234}
	\mathcal{L}_{{3\Pi}} \, f:= \mathcal{L}_3 f+\frac{1}{L}\int_0^L \varphi^2 f \; dx.
\end{equation}
We see that $\mathcal{L}_{{i\Pi}}: D(\mathcal{L}_{{i\Pi}})=H^2_{per} \cap L^2_{per,m}\subset  L^2_{per,m} \rightarrow L^2_{per,m}$ are well defined for $i=2,3$. To determine the spectral properties of  $\mathcal{L}_{{i\Pi}}$ we also use the Index Theorem. 
The index formulas as in $(\ref{indexformula12})$ and $(\ref{indexformula123})$ are now given by
\begin{equation}\label{indexformula1234}
	\text{n}(\mathcal{L}_{{i\Pi}})=\text{n}(\mathcal{L}_i)-n_0-z_0 \qquad \mbox{and} \qquad 	\text{z}(\mathcal{L}_{{i\Pi}})=\text{z}(\mathcal{L}_i)+z_0,
\end{equation}
where for
$
D_i:=(\mathcal{L}_i^{-1} 1,1)_{L^2_{per}}, i=2,3,
$
 the values of $n_0$ and $z_0$ are defined respectively as
\begin{equation*}
	n_0:=
	\begin{cases}
		1, \: \text{if} \: D_i<0, \\
		0, \: \text{if} \: D_i \geq 0,\ \
	\end{cases}
	\quad \text{and} \quad\ 
	z_0:=
	\begin{cases}
		1, \: \text{if} \: D_i=0, \\
		0, \: \text{if} \: D_i \neq 0.
	\end{cases}
\end{equation*}
For the cnoidal waves in $(\ref{Sol3})$, there exists a unique $\omega^*>0$ such that $D_2<0$ for all $\omega\in(0,\omega^*)$. From $(\ref{indexformula1234})$, this gives $n_0=1$ and $z_0=0$, so that $\text{n}(\mathcal{L}_{2\Pi})=1$ and $\text{z}(\mathcal{L}_{2\Pi})=1$. In addition,  we have $D_3<0$, for all $\omega>0$, which gives $\text{n}(\mathcal{L}_{3\Pi})=0$ and $\text{n}(\mathcal{L}_{3\Pi})=1$. Therefore, the full constrained operator $\mathcal{L}_{\Pi}$ associated to $\mathcal{L}$ satisfies $\text{n}(\mathcal{L}_{\Pi})=1$ and $z(\mathcal{L}_{\Pi})=2$, for all $\omega\in (0,\omega^*)$. Having in mind all mentioned restrictions on the zero mean space $L_{per,m}^2$, we obtain that the periodic wave $\varphi$ in $(\ref{Sol3})$ is orbitally stable in the  space $H^1_{per}\cap L_{per,m}^2$ by using an adaptation of the arguments in \cite{NP2015} and \cite{stuart}.

Our paper is organized as follows: in Section \ref{section2} we present some basic notation. In Section \ref{section3}, we show the existence of a family of periodic wave solutions of the cnoidal type for the equations \eqref{KF2} and (\ref{KF1}). Spectral analysis for the linearized operators $\mathcal{L}$ is established in Section \ref{section4}. Finally, the orbital instability/stability of the periodic waves will be shown in Sections \ref{section5} and \ref{section6}.

\section{Notation}\label{section2}

Here we introduce the basic notation concerning the periodic Sobolev spaces. For a more complete introduction to these spaces we refer the reader to \cite{Iorio}. By $L^2_{per}:=L^2_{per}([0,L])$, $L>0$, we denote the space of all square integrable functions which are $L$-periodic. For $s\geq0$, the Sobolev space
$H^s_{per}:=H^s_{per}([0,L])$
is the set of all periodic distributions such that
$$
\|f\|^2_{H^s_{per}}:= L \sum_{k=-\infty}^{\infty}(1+|k|^2)^s|\hat{f}(k)|^2 <\infty,
$$
where $\hat{f}$ is the periodic Fourier transform of $f$. The space $H^s_{per}$ is a  Hilbert space with natural inner product denoted by $(\cdot, \cdot)_{H^s_{per}}$. When $s=0$, the space $H^s_{per}$ is isometrically isomorphic to the space  $L^2_{per}$, that is, $L^2_{per}=H^0_{per}$ (see, e.g., \cite{Iorio}). The norm and inner product in $L^2_{per}$ will be denoted by $\|\cdot \|_{L^2_{per}}$ and $(\cdot, \cdot)_{L^2_{per}}$. In our paper, we do not distinguish if the elements in $H_{per}^s$ are complex- or real-valued. 

For $s \geq 0$, we define
\begin{equation}\label{zeromeandefinition}
H^s_{per,m} := \left\{ f \in H^s_{per}\; ; \; \frac{1}{L}\int_0^L  f(x) \; dx =0 \right\},
\end{equation}
endowed with  norm and inner product of $H_{per}^s$. Denote the topological dual of $H^s_{per,m}$ by
$
H^{-s}_{per,m}:=(H^s_{per,m})' . 
$
In addition, to simplify notation we set
 $$\mathbb{H}^s_{per}:= H^s_{per} \times H^s_{per},\ \
 \mathbb{H}^s_{per,m}:= H^s_{per,m} \times H^s_{per,m},\ \
 \mathbb{L}^2_{per}:= L^2_{per} \times L^2_{per},$$
endowed with their usual norms and scalar products. When necessary and since $\mathbb{C}$ can be identified with $\mathbb{R}^2$, notations above can also be used in the complex/vectorial case in the following sense: for $f\in \mathbb{H}_{per}^s$ we have $f=f_1+if_2\equiv (f_1,f_2)$, where $f_i\in H_{per}^s$ $i=1,2$.\\
\indent Throughout this paper, we  fix the following embedding chains with the Hilbert space $L_{per}^2$ identified with its dual (by the Riesz Theorem) as $$H_{per,m}^1\hookrightarrow H_{per}^1\hookrightarrow L_{per}^2\equiv (L_{per}^2)'\hookrightarrow H_{per}^{-1}\hookrightarrow H_{per,m}^{-1},$$
where $H_{per,m}^{-1}$ stands for the topological dual of $H_{per,m}^1$.

\indent The symbols $\sn(\cdot, k), \dn(\cdot, k)$ and $\cn(\cdot, k)$ represent the Jacobi elliptic functions of \textit{snoidal}, \textit{dnoidal}, and \textit{cnoidal} type, respectively. For $k \in (0, 1)$, $\K(k)$ and $\E(k)$ will denote the complete elliptic integrals of the first and second kind, respectively. For the precise definition and additional properties of the elliptic functions we refer the reader to  \cite{byrd}.

\section{Existence of Periodic Waves of Cnoidal Type}\label{section3}

This section is devoted to prove the existence of cnoidal-type solutions for equations \eqref{KF3} and \eqref{KF4}. Since such results are well-know for similar equations, we only bring  the main steps.

\subsection{Klein-Gordon equation}
Our purpose in this section is to present the existence of
periodic solutions $\varphi_c:=\varphi$ for the nonlinear ODE
\begin{equation}\label{ode1}
-\omega \varphi''+\varphi-\varphi^{3}=0,
\end{equation}
where $\omega:=1-c^2$ and $c \in (-1,1)$. Indeed, by multiplying $(\ref{ode1})$ by $\varphi'$ we see that it reduces to the  quadrature form
\begin{equation}\label{1.EC}
(\varphi')^2=\frac{1}{2 \omega} \left( 2 \varphi^2-\varphi^4+4A \right)=-\frac{1}{2 \omega}F(\varphi),
\end{equation}
where $A \in \mathbb{R}$ is an integration constant and
$
	F(\lambda):= \lambda^4-2\lambda^2-4A.
$
Since $F$ is an even  polynomial, we assume that it has two real roots of the form $\pm \beta_2$ and two purely imaginary ones of the form $\pm i\beta_1$. In this case, we may write \begin{equation}\label{polF2}
	F(\lambda)=(\lambda^2+\beta_1^2)(\lambda^2-\beta_2^2).
\end{equation}
Without loss of generality, we assume  that $\beta_2>0$. In view of \eqref{1.EC} it follows that
$\varphi (x) \in [-\beta_2, \beta_2]$ for all $x\in \mathbb{R}.$ Thus, periodic sign changing solutions with the zero mean property makes sense in this context.

Using  $(\ref{1.EC})$ and $(\ref{polF2})$ we have that $\beta_1$, $\beta_2$ and  $A$ satisfy the relations 
\begin{equation*}
\beta_2^2- \beta_1^2 = 2 \quad
\mbox{ and } \quad \beta_1^2 \beta^2_2= 4A.
\end{equation*}
 We can establish the following existence result of periodic waves for $(\ref{ode1})$.

\begin{proposition}[Smooth Curve of Cnoidal Waves]\label{cnoidalsol}
Let $L>0$ be fixed. For each $c \in (-1,1)$ there exists a unique $\beta_2 \in (\sqrt{2}, \infty)$ such that the cnoidal wave
\begin{equation}\label{cnoidal}
\varphi_c(\xi)= \beta_2 \, \cn(b\,\xi;k),
\end{equation}
with
$$k^2= \frac{\beta_2^2}{2\beta_2^2-2} \quad \text{and} \quad b^2= \frac{\beta^2_2-1}{\omega}=\frac{\beta^2_2-1}{1-c^2},$$
is an $L$-periodic solution of \eqref{ode1}. In addition, the curve $c\in(-1,1)\longmapsto\varphi_c\in H^2_{per}$ is smooth.
\end{proposition}

\begin{proof}
 The proof is based on the Implicit Function Theorem and it is similar to that in \cite[Theorem 2.3]{angulo}; so we omit the details.
\end{proof}

\begin{remark}\label{remarkcnoidalsol}
Let $L>0$ be fixed. For $c \in (-1,1)$, solution in \eqref{cnoidal} can be written  in terms of the modulus $k\in\left(\frac{1}{\sqrt{2}},1\right)$ as in \eqref{Sol3},
that is, $\beta_2$ and $b$ are given by
$$\beta_2= \frac{\sqrt{2}k}{\sqrt{2k^2-1}}\ \ 
\mbox{and}\ \ 	b= \frac{4\K(k)}{L}.
$$
Moreover, the wave speed $c$ is given by $\omega=1-c^2$, with $\omega$ as in \eqref{ccn}.
\end{remark}

\subsection{Schr\"odinger equation}
Our goal here is to find explicit solutions for the ODE
\begin{equation}\label{ode2}
- \varphi''+\omega \varphi-\varphi^{3}=0.
\end{equation}
In a similar fashion as Proposition \ref{cnoidalsol}, we can infer the following result.

\begin{proposition}[Smooth Curve of Cnoidal Waves]\label{cnoidalsol2}
Let $L>0$ be fixed. For each $\omega >0$ there exists a unique $\beta_3 \in (\sqrt{2\omega}, \infty)$ such that the cnoidal wave
\begin{equation}\label{cnoidal2}
\varphi_{\omega}(x)= \beta_3 \, \cn(b_1\, x;k),
\end{equation}
with
$$k^2= \frac{\beta_3^2}{2\beta_3^2-2\omega} \quad \text{and} \quad b^2_1= \beta^2_3- \omega,$$
is an $L$-periodic solution of \eqref{ode2}. In addition, the curve $\omega\in(0,+\infty)\longmapsto\varphi_\omega\in H^2_{per}$ is smooth.
\end{proposition}

\begin{remark}\label{remarkcnoidalsol1}
It can also be viewed that solution in \eqref{cnoidal2} may be written as  in \eqref{Sol3},
	that is, $\beta_3$ and $b_1$  are given by
	
	$$\beta_3= \frac{\sqrt{2\omega}k}{\sqrt{2k^2-1}},\ \ 
 	b_1= \frac{4\K(k)}{L},
	$$
and the frequency $\omega$ is given by	\eqref{omegacn2}.
\end{remark}

\section{Spectral Analysis.}\label{section4}


\subsection{Spectral analysis for the Klein-Gordon equation}\label{spectralKG}  Throughout this subsection, we fix $L>0$ and, for $c \in (-1,1)$, we let $\varphi=\varphi_c$ be the periodic solution of $(\ref{ode1})$ given by Proposition \ref{cnoidalsol}. The aim of this section is to study the non-positive spectrum of the linearized operator $\mathcal{L}$ defined in \eqref{matrixop313}.  As we already said, the spectral information of $\mathcal{L}$ is closely related to that of $\mathcal{L}_1$, defined in \eqref{linoperat}. Indeed,  the quadratic form associated to $\mathcal{L}$ reads as
\begin{equation*}
\begin{split}
		Q(u,v)&=\displaystyle\left(\mathcal{L}(u,v),(u,v)\right)_{\mathbb{L}^2_{per}} \\
		&=\displaystyle\int_{0}^{L} \left(\omega u'^2+u^2-3\varphi^2u^2\right)  dx+\big\|cu'-v\big\|_{L^2_{per}}^2\\
		&=Q_1 (u)+\big\|c u'-v\big\|_{L^2_{per}}^2,
\end{split}
\end{equation*}
where, for $\omega=1-c^2>0$, 
\begin{equation*}
	Q_1 (u) :=\int_{0}^{L} \left(\omega u'^2+u^2-3\varphi^2u^2 \right) dx
\end{equation*}
represents the quadratic form associated to $\mathcal{L}_1$. So the idea is  to first count the non-positive spectrum of $\mathcal{L}_1$. After that, we proceed in counting the negative eigenvalues of $\mathcal{L}$ and prove that $\ker(\mathcal{L})=[(\varphi',c\varphi'')]$.

To start with, from $(\ref{ode1})$ we promptly see that $\varphi' \in \Ker(\mathcal{L}_1)$. In addition, since $\varphi'$ has exactly two zeros in the half-open interval $[0,L )$ we see that zero is the second or the third eigenvalue of  $\mathcal{L}_1$ (see, for instance, \cite[Theorem 3.1.2]{eas}). The next result gives that zero is indeed the third one.

\begin{lemma}\label{2eigenvalues}
The operator $\mathcal{L}_1$ in \eqref{linoperat} has exactly two negatives eigenvalues which are simple. Zero is a simple eigenvalue with associated eigenfunction $\varphi'$. In addition, the rest of the
spectrum is constituted by a discrete set of eigenvalues.
\end{lemma}
\begin{proof}
See \cite[Theorem 3.2]{angulo}.
\end{proof}

\begin{remark}\label{1eigenvalue}
From Lemma \ref{2eigenvalues}, if $h \in H^2_{per} $ is an eigenfunction associated to the negative eigenvalue $-\kappa^2$ of $\mathcal{L}_1$ then
$$Q(h, c h')=Q_1(h)=(\mathcal{L}_1 h, h)_{L^2_{per}}= -\kappa^2 \big\|h\big\|^2_{L^2_{per}}<0,$$
which implies that the number of negative eigenvalues of $\mathcal{L}$ is at least $1$. 
\end{remark}

\indent Next result gives the precise spectral information of the operator $\mathcal{L}_{1\Pi}$ in $(\ref{opconstrained12})$. Here and in what follows, we only consider the case $c\in[0,1)$.

\begin{lemma}\label{leman1}
 There exists $c^{*}\in (0,1)$ such that:
	\begin{itemize}
		\item[(i)] for all $c \in [0,c^*)$, operator 
		$
		\mathcal{L}_{{1\Pi}}
		$
		has exactly one negative eigenvalue which is simple. Zero is a simple eigenvalue with eigenfunction $\varphi'$;
		\item[(ii)]  For all $c \in (c^*,1)$, operator 
		$
		\mathcal{L}_{{1\Pi}}
		$ has exactly two negative eigenvalues which are simple. Zero is a simple eigenvalue with eigenfunction $\varphi'$;
		\item[(iii)] If $c=c^*$, operator $\mathcal{L}_{1\Pi}$ has one negative eigenvalue which is simple and zero is a double eigenvalue.
	\end{itemize}
\end{lemma}
\begin{proof}
This  was essentially proved in \cite{kapitulodeco}; but here we give a slightly different proof.
Since the function 
	$$
	k\in \left(\frac{1}{\sqrt{2}},1\right)\mapsto c=\sqrt{1-\frac{1}{16}\frac{L^2}{\K(k)^2 (2k^2-1)}}\in (0,1)
	$$
	is increasing, it suffices to work with the parameter $k$ instead of $c$.
	
	 The first five eigenvalues and eigenfunctions of $\mathcal{L}_1$ are well-known (see, for instance, \cite{angulo} or \cite{kapitulodeco}). In particular, the first and fifth eigenvalues are
$$
\lambda_0:= \frac{1-2k^2-2\sqrt{1-k^2+k^4}}{2k^2-1} \quad \text{and} \quad \lambda_4:= \frac{1-2k^2+2\sqrt{1-k^2+k^4}}{2k^2-1}, 
$$
with respective eigenfunctions
$$f_0(x):=k^2 \sn^2(bx;k)-\frac{1}{3}(1+k^2+\sqrt{1-k^2+k^4})$$
and
$$f_4(x):=k^2 \sn^2(bx;k)-\frac{1}{3}(1+k^2-\sqrt{1-k^2+k^4}).$$
By observing that
$$f_0(x)-f_4(x)= -\frac{2}{3} \sqrt{1-k^2+k^4}.$$
we obtain
$$
\mathcal{L}_1(\lambda_4f_0-\lambda_0 f_4)= -\frac{\lambda_0 \lambda_4}{3}\sqrt{1-k^2+k^4},
$$
and
\begin{equation}\label{inverseL1}
-\frac{3}{2 \sqrt{1-k^2+k^4}} \mathcal{L}_1 \left( \frac{\lambda_4 f_0 - \lambda_0 f_4}{\lambda_0 \lambda_4}\right)=1.
\end{equation}
Therefore,
\begin{equation*}
	D_1=(\mathcal{L}_1^{-1} 1, 1 )_{L^2_{per}} =-\frac{3}{2} \frac{\lambda_4(f_0,1)_{L^2_{per}}-\lambda_0(f_4,1)_{L^2_{per}}}{\sqrt{1-k^2+k^4} \lambda_0 \lambda_4}.
\end{equation*}
Using the explicit form of $f_0$ and $f_4$, we are able to compute the above inner products to get
\begin{equation}\label{D1}
	D_1= -\frac{L(2k^2-1)}{\K(k)}\left( 2\E(k)-\K(k)\right).
\end{equation}
Since $-\frac{L(2k^2-1)}{\K(k)}<0$ for all $k \in \left(\frac{1}{\sqrt{2}},1\right),$
we have that $D_1=0$ if and only if $2\E(k)-\K(k)=0$. This is achieved at a unique $k=k^{*}\approx 0.908$. In addition, since $k\mapsto2\E(k)-\K(k)$ is a decreasing function, we deduce
\begin{equation*}
	\left\{
	\begin{array}{l}
D_1<0, \quad \mbox{if}\;\; k \in \left(\frac{1}{\sqrt{2}},k^*\right), \\
D_1>0, \quad \mbox{if} \;\; k \in \left(k^*,1\right).
\end{array}
\right.
\end{equation*}

For $k \in \left(\frac{1}{\sqrt{2}},k^*\right)$, from \eqref{indexformula12}, \eqref{indexformula123}, and Lemma \ref{2eigenvalues} we obtain
\begin{equation*}
\text{n}(\mathcal{L}_{1\Pi})=2-1-0=1 \qquad \mbox{and} \qquad  \text{z}(\mathcal{L}_{1\Pi})=1+0=1.
\end{equation*}
that is, $\mathcal{L}_{1\Pi}$ has only one negative eigenvalue and zero is a simple eigenvalue. These facts establish (i). On the other hand, for  $k \in \left(k^*,1\right)$ we obtain
$$
	\text{n}(\mathcal{L}_{1\Pi})=2-0-0=2 \qquad \mbox{and} \qquad 	\text{z}(\mathcal{L}_{1\Pi})=1+0=1,
$$
which gives (ii). Finally, at $k=k^*$,
$$
\text{n}(\mathcal{L}_{1\Pi})=2-0-1=1 \qquad \mbox{and} \qquad  \text{z}(\mathcal{L}_{1\Pi})=1+1=2,
$$
which yields (iii) and completes the proof of the lemma.
\end{proof}

Let us turn our attention to the operator $\mathcal{L}$ defined in \eqref{matrixop313}. We first prove the following.

\begin{lemma}\label{lema00}
For any $c\in[0,1)$ we have $\Ker(\mathcal{L})=[(\varphi',c\varphi'')]$.
\end{lemma}
\begin{proof}
Note that $(f,g)$ belongs to $\Ker(\mathcal{L})$ if and only if
\[
\begin{cases}
-f''+f-3\varphi^2f+cg'=0,\\
-cf'+g=0.
\end{cases}
\]
By substituting the second equation in the first one we see that $f$ must belong to $\Ker(\mathcal{L}_1)$. Hence, the result follows from Lemma \ref{2eigenvalues}.
\end{proof}

\begin{proposition}   \label{teoeigenKG}
 For $c \in [0, 1)$, the operator 
$
\mathcal{L}
$ has two negative eigenvalues which are simple and zero is the third eigenvalue with eigenfunction $(\varphi',c\varphi'')$. Moreover, the rest of the spectrum is constituted by a discrete set bounded away from zero. 
\end{proposition}
\begin{proof}
For $c=0$, we have by $(\ref{matrixop313})$ 

\begin{equation}\label{matrixop3133}
	\displaystyle \mathcal{L}=\left(
	\begin{array}{ccc}
		-\partial_x^2-3\varphi^2+1  & &0\\\\
		\ \ \ \ 0 & & 1
	\end{array}\right)=\left(
\begin{array}{ccc}
\mathcal{L}_1  & &0\\\\
 0 & & 1
\end{array}\right).
\end{equation}
Since operator $\mathcal{L}$ is diagonal and by Lemma $\ref{2eigenvalues}$ we have $n(\mathcal{L}_1)=2$, it follows by Lemma $\ref{lema00}$ and the continuity of the eigenvalues in terms of $c\in[0,1)$ that $n(\mathcal{L})=2$ for all $c\in[0,1)$.

\end{proof}

In order to count the number of negative eigenvalues of $\mathcal{L}_{\Pi}$ in $(\ref{opconstrained2})$, we need to use Lemma $\ref{leman1}$. First, we see that $\mathcal{L}_{\Pi}$ is the constrained operator $\mathcal{L}$ defined in $L_{per,m}^2\times L_{per}^2$ with constrained space $S:= [(1,0),(0,1)] \subset \Ker(\mathcal{L})^{\perp}=[(\varphi',c\varphi'')]^\perp$.

\begin{proposition}\label{teoeigenKGconst} There exists $c^{*}\in (0,1)$ such that:
	\begin{itemize}
		\item[(i)] for all $c \in [0,c^*)$, operator 
		$
		\mathcal{L}_{{\Pi}}
		$
		has exactly one negative eigenvalue which is simple. Zero is a simple eigenvalue with eigenfunction $(\varphi',c\varphi'')$;
		\item[(ii)]  For all $c \in (c^*,1)$, operator 
		$
		\mathcal{L}_{{\Pi}}
		$ has exactly two negative eigenvalues which are simple. Zero is a simple eigenvalue with eigenfunction $(\varphi',c\varphi'')$;
		\item[(iii)] If $c=c^*$, operator $\mathcal{L}_{\Pi}$ has one negative eigenvalue which is simple and zero is a double eigenvalue.
	\end{itemize}
\end{proposition}

\begin{proof} For
$
\tilde{f}= -\frac{3}{2 \sqrt{1-k^2+k^4}} \left( \frac{\lambda_4 f_0 - \lambda_0 f_4}{\lambda_0 \lambda_4}\right)$, we have
 by   \eqref{inverseL1} that $\mathcal{L}_1^{-1}1=\tilde{f}$. On the other hand, associated to the constrained set $S$, let us define the matrix
 
 \begin{equation}\label{matrixD}
 	D:=\left[\begin{array}{llll}(\mathcal{L}^{-1}(1,0),(1,0))_{\mathbb{L}^2_{per}}& & (\mathcal{L}^{-1}(1,0),(0,1))_{\mathbb{L}^2_{per}}\\\\
 		(\mathcal{L}^{-1}(1,0),(0,1))_{\mathbb{L}^2_{per}}& & (\mathcal{L}^{-1}(0,1),(0,1))_{\mathbb{L}^2_{per}}\end{array}\right].
 	\end{equation}
Since $\mathcal{L}(0,1)=(0,1)$ and $(1,1)=(1,0)+(0,1)$, we have $\mathcal{L}^{-1}(1,0)=\mathcal{L}^{-1}(1,1)-\mathcal{L}^{-1}(0,1)=\mathcal{L}^{-1}(1,1)-(0,1)$ and
 \begin{equation}\label{D}
\left( \mathcal{L}^{-1}(1,0), (1,0) \right)_{\mathbb{L}^2_{per}}= (\tilde{f},1)_{L^2_{per}}  = ( \mathcal{L}_1^{-1} 1,1)_{L^2_{per}}=D_1.
\end{equation}
In addition, 
\begin{equation}\label{D12}
		\left( \mathcal{L}^{-1}(1,0), (0,1) \right)_{\mathbb{L}^2_{per}}=\left(\mathcal{L}^{-1}(1,1),(0,1)\right)_{\mathbb{L}^2_{per}}-\left(\mathcal{L}^{-1}(0,1),(0,1)
		\right)_{\mathbb{L}^2_{per}}=L-L=0,\end{equation}
and 
\begin{equation}\label{D123}
	\left( \mathcal{L}^{-1}(0,1), (0,1) \right)_{\mathbb{L}^2_{per}}=\left((0,1),(0,1)\right)_{\mathbb{L}^2_{per}}=L.\end{equation}
From \eqref{D}, \eqref{D12} and \eqref{D123}, we obtain that $D$ in \eqref{matrixD} is now given by $D=\left[\begin{array}{cc}D_1&  0\\
	0&  L\end{array}\right]$. Since $\det(D)=D_1L$, the result follows by a direct application of Lemma $\ref{leman1}$ and the Index Theorem.

\end{proof}

\subsection{Spectral analysis for the Schr\"odinger equation} \label{EspectralAnalysisSchrodinger}

Let $L>0$ be fixed. Throughout this subsection, for $\omega>0$, we let $\varphi=\varphi_{\omega}$ be the periodic solution of $(\ref{ode2})$ given by Proposition \ref{cnoidalsol2}. As we already mention, we will prove an stability result in the space of zero mean functions. Therefore, here we study the spectrum of the constrained operator $\mathcal{L}_{\Pi}$ defined by
\begin{equation*}
	\mathcal{L}_{\Pi}=\left(\begin{array}{cccc}\mathcal{L}_{2\Pi}& & 0\\
		0& & \mathcal{L}_{3\Pi}\end{array}\right),
	\end{equation*}
where $\mathcal{L}_{2\Pi}$ and $\mathcal{L}_{3\Pi}$ are defined by $(\ref{opconstrained123})$ and $(\ref{opconstrained1234})$, respectively.

First of all, we recall that operators $\mathcal{L}_2$ and $\mathcal{L}_3$ in $(\ref{L2L3})$ satisfy  $\text{n}(\mathcal{L}_2)=2$, $\text{n}(\mathcal{L}_3)=1$ and $\text{z}(\mathcal{L}_2)=\text{z}(\mathcal{L}_3)=1$ (see \cite[Theorems 3.2 and 3.4]{angulo}).
So, repeating the steps as in Lemma \ref{leman1} with $\mathcal{L}_{2\Pi}$ in place of $\mathcal{L}_{1\Pi}$, we  infer the existence of $k^* \simeq 0.908$ such that
$(\mathcal{L}_2^{-1} 1,1)_{L^2_{per}} <0$ for all $k \in \left( \frac{1}{\sqrt{2}},k^* \right),$ $(\mathcal{L}_2^{-1} 1,1)_{L^2_{per}} =0$ for $k=k^*$, and $(\mathcal{L}_2^{-1} 1,1)_{L^2_{per}} >0$ for all $k\in(k^*,1)$.
Applying  the Index Theorem we then deduce that $\text{n}(\mathcal{L}_{2\Pi})=1$ for all $k\in\left(\frac{1}{\sqrt{2}},k^*\right]$ and $\text{n}(\mathcal{L}_{2\Pi})=2$ for all $k\in(k^*,1)$. In addition $\text{z}(\mathcal{L}_{2\Pi})=1$ for $k\neq k^*$ and $\text{z}(\mathcal{L}_{2\Pi})=2$ at $k=k^*$.

Summarizing the arguments above, we have the following result.

\begin{proposition}\label{propn1}
 There exists $\omega^{*}>0$ such that:
 \begin{itemize}
 	\item[(i)] for all $\omega \in (0,\omega^*)$, operator 
 	$
 	\mathcal{L}_{{2\Pi}}
 	$
 	has exactly one negative eigenvalue which is simple. Zero is a simple eigenvalue with eigenfunction $\varphi'$;
 	\item[(ii)] For all $\omega \in (\omega^*,+\infty)$, operator 
 	$
 	\mathcal{L}_{{2\Pi}}
 	$ has exactly two negative eigenvalues which are simple. Zero is a simple eigenvalue with eigenfunction $\varphi'$;
 	\item[(iii)] If $\omega=\omega^*$, operator $\mathcal{L}_{2\Pi}$ has one negative eigenvalue which is simple and zero is a double eigenvalue.
 \end{itemize}
\end{proposition}

To apply the Index Theorem for $\mathcal{L}_{3\Pi}$, we implement a different approach from the one in Lemma $\ref{leman1}$. Indeed, we are not able to calculate the quantity $D_3:=(\mathcal{L}_3^{-1}1,1)_{L^2_{per}}$ using the explicit expressions of eigenvalues and eigenfunctions associated to the linear operator $\mathcal{L}_3$. To overcome this difficulty, we proceed as follows: let $y$ be the solution of the following initial-value problem (IVP)
\begin{equation}\label{Cauchyproblem} 
	\begin{cases}
		-y''+\omega y-\varphi^2 y=0, \\
		y(0)= 0, \\
		y'(0)=-\frac{1}{\varphi(0)}.
	\end{cases}
\end{equation}
 Using $(\ref{Cauchyproblem})$ and noting that the Wronskian of the solutions $\varphi$ and $y$ associated with the equation $(\ref{Cauchyproblem})$ is constant by Abel's formula, we can assume that the Wronskian of $y$ and $\varphi$ is equal to 1. This implies that $\{y,\varphi\}$ forms a fundamental set of solutions for the equation in \eqref{Cauchyproblem}. In addition,  $y$ is an odd function and there exists a constant $\theta$ such that (see, for instance, \cite[pages 4 and 8]{magnus}) 
\begin{equation}\label{thetadef}
	y(x+L)=y(x)+\theta \varphi(x), \qquad x\in\R.
\end{equation}
From the elementary theory of ODE's, it is well known that 
$$
p(x):=-\left(\int_0^xy(s)ds\right)\varphi(x)+\left(\int_0^x\varphi(s)ds\right)y(x)
$$ 
satisfies
$$
	-p''+\omega p-\varphi^2 p=1.
$$

We shall show that $p$ is indeed an $L$-periodic function. First of all note that \eqref{thetadef} and the fact that $y$ is odd  imply that $\int_0^Ly(x)dx=0$.  Thus, $$p(L)=-\left(\int_0^Ly(x)dx\right)\varphi(L)+\left(\int_0^L\varphi(x)dx\right)y(L)=0=p(0),$$
where in the second integral, we are using the fact that $\varphi$ has the zero mean property. Also, since
$$\begin{array}{lllll}\displaystyle p'(x)&=&\displaystyle-y(x)\varphi(x)-\left(\int_0^xy(s)ds\right)\varphi'(x)+\varphi(x)y(x)+\left(\int_0^x\varphi(s)ds\right)y'(x).\\\\
&=&-\displaystyle\left(\int_0^xy(s)ds\right)\varphi'(x)+\left(\int_0^x\varphi(s)ds\right)y'(x)\end{array}$$
we obtain $p'(L)=p'(0)=0$ and $p$ is $L$-periodic as desired. In particular we see that  $p$ satisfies the IVP
\begin{equation}\label{Cauchyproblem1} 
	\begin{cases}
		-p''+\omega p-\varphi^2 p=1, \\
		p(0)= 0, \\
		p'(0)=0.
	\end{cases}
\end{equation}
Problem \eqref{Cauchyproblem1} is suitable enough to perform numeric calculations. In fact, solving numerically  \eqref{Cauchyproblem1}, we conclude $D_3=(\mathcal{L}_3^{-1}1,1)_{L^2_{per}}=(p,1)_{L^2_{per}}<0$ for all $k\in\left(\frac{1}{\sqrt{2}},1\right)$, and hence for all $\omega>0$. An illustration of the behaviour of $D_3$, for  $L=2\pi$, appears in  Figure $4.1$.

\begin{figure}[h]
	\centering 
	\includegraphics[scale=0.3]{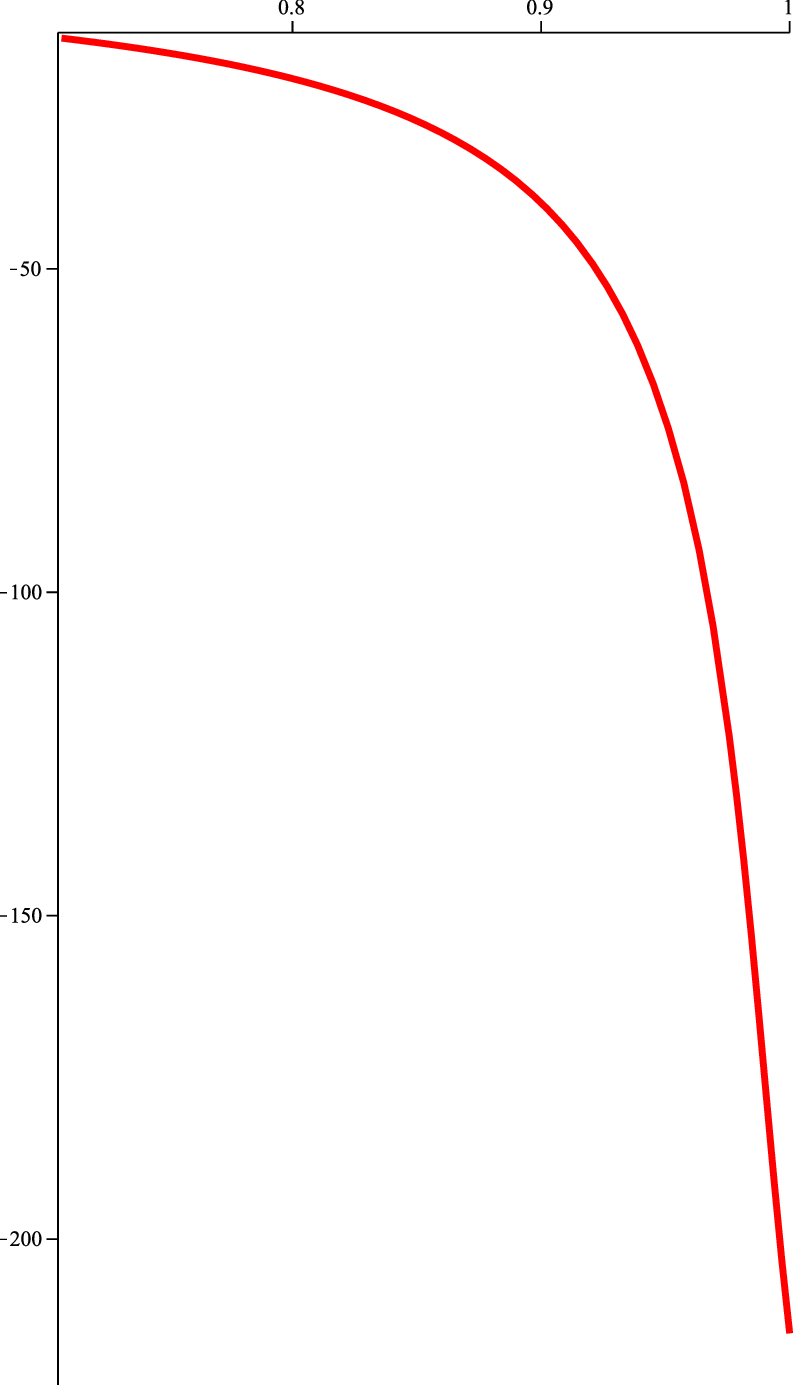} 
	\caption{Graphic of $D_3$ for $L=2\pi$.}
\end{figure}

\noindent The Index Theorem, (see  (\ref{indexformula1234})) finally yields
\begin{equation*}
		\text{n}(\mathcal{L}_{3\Pi}) = 1 - 1 - 0 = 0 \quad \mbox{and}\quad  	\text{z}(\mathcal{L}_{3\Pi})=1+0=1.
\end{equation*}
Hence, we have proved the following result.

\begin{proposition}\label{propn0}
 For any $\omega \in (0,+\infty)$, the operator 
$
\mathcal{L}_{{3\Pi}} 
$
defined in \eqref{opconstrained1234}  has no  negative eigenvalue. In addition, zero is a simple eigenvalue with eigenfunction $\varphi$.
\end{proposition}

\begin{remark}
	It is important to mention that the behaviour of $D_3$ in Figure $4.1$ is given only for $2\pi-$periodic functions. A scaling argument can be used to justify that $D_3<0$ for all values of $L>0$. In fact, consider $\widetilde{\varphi}(z)=\frac{L}{\widetilde{L}}\varphi\left(\frac{L}{\widetilde{L}}z\right)$, where $\varphi$ solves $(\ref{KF4})$. Since $\varphi$ is $L-$periodic, we obtain that $\widetilde{\varphi}$ is $\widetilde{L}-$periodic and solves the equation $(\ref{KF4})$ with $\widetilde{\omega}:=\left(\frac{L}{\widetilde{L}}\right)^2\omega$ in place of $\omega$. By considering the new solution $\widetilde{\varphi}$ instead of $\varphi$ in the initial value problem $(\ref{Cauchyproblem})$, we obtain the non-periodic solution $\widetilde{y}(z)=y\left(\frac{L}{\widetilde{L}}z\right)$ that solves  the problem 
	\begin{equation}\label{Cauchyproblemtilde} 
		\begin{cases}
			-\widetilde{y}''+\widetilde{\omega}\widetilde{y}-\widetilde{\varphi}^2 \widetilde{y}=0, \\
			\widetilde{y}(0)= 0, \\
			\widetilde{y}'(0)=-\frac{L}{\widetilde{L}}\frac{1}{\varphi(0)}.
		\end{cases}
	\end{equation}
	The functions $\widetilde{y}(z)$ and $\widetilde{\varphi}$ are also related to each other through equality $(\ref{thetadef})$ with $\widetilde{L}$ in place of $L$ and $\widetilde{\theta}$ in place of $\theta$. Thus, by defining $$\widetilde{p}(z):=-\left(\int_0^z\widetilde{y}(t)dt\right)\widetilde{\varphi}(z)+\left(\int_0^z\widetilde{\varphi}(t)dt\right)\widetilde{y}(z),$$
	we obtain that $\widetilde{p}$  satisfies the problem $(\ref{Cauchyproblem1})$ with $\widetilde{\omega}$ in place of $\omega$ and $\widetilde{p}(z)=p\left(\frac{L}{\widetilde{L}}z\right)$. Doing the calculations as done before and considering $\widetilde{L}=2\pi$, we obtain $\int_0^Lp(t)dt=\frac{L}{2\pi}\int_{0}^{2\pi}\widetilde{p}(x)dx<0$ as requested.
\end{remark}

\section{Orbital instability of the cnoidal waves for the KG equation}\label{section5}

The goal of this section is to establish a result of orbital instability based on the classical theory contained in \cite{grillakis1} for the periodic  traveling wave solutions $\varphi$ in $(\ref{cnoidal})$  for the KG equation \eqref{KF2}.  Consider the restricted energy space $X:= H_{per,m}^1 \times L_{per,m}^2$. It is well known that \eqref{KF2} is invariant by translations. Thus, we can define for $x,r \in \mathbb{R}$ and $U=(u,v) \in X$  
 \begin{equation*}
\mathrm{T}_rU(x)=(u(x+r), v(x+r)). 
\end{equation*}
In what follows,  for $c\in(-1,1)$, we define
$$
\Phi:=(\varphi, c\varphi').
$$

Next we recall the definition of the orbital stability in this context.

\begin{definition}[Orbital Stability]\label{stadef}
The periodic wave $\Phi$ is said orbitally stable  if for all $\varepsilon>0$ there exists $\delta>0$ with the following property: if
$$
\|(u_0, v_0)-(\varphi, c\varphi')\|_{X}<\delta
$$
and $U=(u,v):=(u,u_t)$ is the weak solution of \eqref{KF2} in the interval $[0,t_0)$, for some $t_0>0$, with $U(0)=(u_0,v_0)$, then $U$ can be continued to a solution in $0 \leq t < \infty$ and
$$
\sup_{t \in \mathbb{R}_+}\inf_{s \in \mathbb{R}}\big\|U(t)-\mathrm{T}_s \Phi \big\|_X< \varepsilon.
$$

Otherwise,  $\Phi$ is said to be  orbitally unstable. In particular, this would happen in the case of solutions that cannot be continued.

\end{definition}


Before establishing our orbital instability result, we present some additional results concerning the existence of weak solutions for Cauchy problem associated with \eqref{KF2} using the  well known \textit{Potential Well Theory} (for further details, see \cite[Chapter 1]{lions1} and \cite{vitillaro}). First of all, for $\omega=1-c^2>0$ fixed, consider 
\begin{equation}\label{Ju}
	P_{\omega}(u)=\frac{1}{2}\int_0^L(\omega u_x^2+u^2)dx-\frac{1}{4}\int_0^Lu^4dx,
	\end{equation}
where $u\in H_{per,m}^1$. For a fixed $L>0$, the critical points of $P_{\omega}$ in $H_{per,m}^1$ are unique and  given in terms of cnoidal solutions as in $(\ref{cnoidal})$. Related to the functional $P_{\omega}$ in $(\ref{Ju})$, 
we have the Nehari manifold
\begin{equation*}\label{nehari}
	\mathcal{N}=\left\{u\in H_{per,m}^1\backslash\{0\};\ \int_0^L (\omega u_x^2+u^2)dx=\frac{1}{2}\int_0^Lu^4dx\right\}.
	\end{equation*}
It is not difficult to show that $P_{\omega}$ satisfies the assumptions of the Mountain Pass Theorem (see e.g. \cite{willem}). In addition, 
\begin{equation*}\label{d}
	d_{\omega}:=\inf_{v\in H_{per,m}^1\backslash\{0\}}\max_{t\geq0}P_{\omega}(tv),
	\end{equation*}
is a critical level for the functional $P_{\omega}$. By using classical results in \cite[Chapter 4]{willem}, we obtain 
\begin{equation}\label{d1}
	d_{\omega}=\min_{u\in \mathcal{N}}P_{\omega}(u)=P_{\omega}(\varphi_{\omega}).
	\end{equation}
\indent Concerning the cnoidal waves in $(\ref{Sol2})$, we obtain after some calculations that
\begin{equation}\label{Jpsi}
	P_{\omega}(\varphi_{\omega})=-\frac{L(K(k)-2)\left( \left(k^4-\frac{5}{3}k^2+\frac{2}{3}\right)K(k)+E(k)\left(\frac{4}{3}k^2-\frac{2}{3}\right)\right)}{K(k)(2k^2-1)^2}=\sqrt{\omega}P_{1}(\varphi_1),
	\end{equation}
where $k\in\left(\frac{1}{\sqrt{2}},1\right)$. In addition, by \eqref{Jpsi} it follows that $P_{\omega}(\varphi_{\omega})=0$ if and only if $K(k)-2=0$, that is, if and only if 
\begin{equation}
	\label{k1}k=k_1:\approx 0.802.
	\end{equation} 
Since $\lim_{k\rightarrow \frac{1}{\sqrt{2}}^{+}}P_{\omega}(\varphi_{\omega})=+\infty$, we see for a fixed $k\in \left(\frac{1}{\sqrt{2}},k_1\right)$ that $P_{\omega}(\varphi_{\omega})>0$ is finite. As a consequence, by $(\ref{d1})$ we have $P_{\omega}(\varphi_{\omega})\leq P_1(\varphi_1)$, so that $d_{\omega}=\sqrt{\omega}d_1\leq d_1$ for all $k\in\left(\frac{1}{\sqrt{2}},k_1\right)$ .\\
\indent In this case, let us define
\begin{equation}\label{W}
	W_{\omega}=\{u\in H_{per,m}^1;\ P_{\omega}(u)<d_{\omega}\}
	\end{equation}
and
\begin{equation}\label{W1}
	W_{1,\omega}=\left\{u\in W_{\omega};\ \int_0^L(\omega u_x^2+u^2)dx>\frac{1}{2}\int_0^Lu^4dx\right\}\cup\{0\}.
	\end{equation}
We have the following well-posedness result.	
\begin{theorem}\label{wellposedness1}
The Cauchy problem associated with \eqref{KF2} is globally well-posed in $X$ in the following sense: for each initial data $u_0\in W_{1,\omega}$ and $v_0\in L_{per,m}^2$ satisfying
\begin{equation}\label{estexist1}\frac{1}{2}||v_0||_{L_{per}^2}^2+P_1(u_0)<d_{\omega},
\end{equation}
 there exists a unique (weak) solution $u$ for the equation \eqref{KF2} in the sense that
 $$
 	\langle u_{tt}(\cdot,t), \phi\rangle_{H^{-1}_{per,m},H_{per,m}^1} + \int_{0}^{L} u_x(\cdot,t)\phi_x \;dx+\int_{0}^{L} u(\cdot,t)\phi \;dx - \int_{0}^{L} u(\cdot,t)^3\phi\;dx = 0, \, \mbox{ a.e. } t \in [0,T],
$$
 for all $\phi \in H^1_{per,m}$. Moreover, we have that $u\in C([0,T]; H_{per,m}^1)$ and $u_t\in C([0,T];L_{per,m}^2)$ for all $T>0$.
\end{theorem}
\begin{proof}
By using Galerkin's method, the proof can be done using similar arguments in \cite[Chapter 1, page 32]{lions1} and it gives the existence of a weak solution $u\in L^{\infty}(0,T; H_{per,m}^1)$ and $u_t\in L^{\infty}(0,T;L_{per,m}^2)$ for all $T>0$. In fact, for the case $\omega=1$ the proof follows identical. For the case $\omega\in(0,1)$, we need  some slight modifications. We can determine an a priori estimate to establish the existence of global solutions as
\begin{equation}\label{bound1}\begin{array}{lllll}
	\displaystyle\frac{1}{2}||u_x(t)||_{L_{per}^2}^2&=&\displaystyle E(u_0,v_0)-\displaystyle\frac{1}{2}||u(t)||_{L_{per}^2}^2-\frac{1}{2}||u_t(t)||_{L_{per}^2}^2+\frac{1}{4}||u(t)||_{L_{per}^4}^4\\\\
	&<&\displaystyle E(u_0,v_0)+\frac{\omega}{2}||u_x(t)||_{L_{per}^2}^2,
	\end{array}\end{equation}
where we are using the fact that $u_0\in W_{1,\omega}$ implies that $u(t)\in W_{1,\omega}$ a.e. $t\geq0$. Since $||u_x(t)||_{L_{per}^2}< \frac{1}{\sqrt{1-\omega}}\sqrt{2E(u_0,v_0)}$ a.e. $t\geq0$, it follows by the Poincar\'e-Wirtinger inequality and the Sobolev embedding $H_{per}^1\hookrightarrow L_{per}^4$ that $u\in L^{\infty}(0,T; H_{per,m}^1)$ and $u_t\in L^{\infty}(0,T;L_{per,m}^2)$ for all $T>0$. The uniqueness can be obtained by standard arguments and a Gronwall-type inequality. This fact enables us to conclude that $E(u,u_t)=E(u_0,v_0)$ and $F(u,u_t)=F(u_0,v_0)$ for all $t\in[0,T]$, so that both $u$ and $u_t$ are continuous in time as required. 
\end{proof}

\begin{remark}\label{wellposednessRemark}
It is important to highlight that Theorem $\ref{wellposedness1}$ can be considered in a smoother space. In fact according to  \cite[Remark 2.3, page 34]{lions1} it is possible to consider $u_0\in  H_{per}^2\cap W_{1,m}$ and $v_0\in H_{per,m}^1$ satisfying $(\ref{estexist1})$ in order to obtain (using Galerkin's method) that the solution $u$ of the Cauchy problem for the equation $(\ref{KF2})$ satisfies $u\in L^{\infty}(0,T; H_{per}^2\cap H_{per,m}^1)$, $u_t\in L^{\infty}(0,T;H_{per,m}^1)$ and $u_{tt}\in L^{\infty}(0,T;L_{per,m}^2)$. Thus, strong solutions in $H_{per}^2\cap H_{per,m}^1\times H_{per,m}^1$ can be considered in this case and it makes sense to see equation $(\ref{KF2})$ restricted to the space constituted by zero mean solutions. Therefore, we are enabled to consider the problem in determining the orbital stability of periodic waves with the zero mean property using the abstract approach in \cite{grillakis1} in the periodic context. Our approach for the equation $(\ref{KF2})$ follows in this direction.
\end{remark}

\subsection{Instability for the cnoidal wave.} 
\indent Let $L>0$ be fixed. For $c\in (-1,1)$, let us consider the cnoidal wave $\varphi$ given by $(\ref{cnoidal})$. The existence of the smooth curve $c \in I=(-1,1)\longmapsto\varphi_c$ given in Proposition \ref{cnoidalsol} allows to define the function $\mathsf{d}: I \subset \mathbb{R}\longrightarrow \mathbb{R} $ given by 
$\mathsf{d}(c)=E(\varphi,c\varphi')-cF(\varphi,c\varphi')$.
 As we already pointed out in the introduction, we have
\begin{eqnarray}\label{5.0}
\begin{array}{lllll}\mathsf{d}''(c)
=\displaystyle-\int_0^{L}(\varphi'(x))^2dx+2(1-\omega)\frac{d}{d\omega}\int_0^{L}(\varphi'(x))^2dx.
\end{array}
\end{eqnarray}

Now, using the explicit expressions in Remark $\ref{remarkcnoidalsol1}$ and formulas $(312.02),$ $(312.04)$ and $(361.02)$ of \cite{byrd}, we obtain that
\begin{eqnarray*}
\int_0^{L}(\varphi'(x))^2dx  &=&  \frac{2k^2}{2k^2-1}\frac{4\K(k)}{L}\int_{0}^{4\K(k)} {\rm sn}^2(u,k){\rm dn}^2(u,k)\;du \\
& = & \frac{32\K(k)}{3(2k^2-1)L}\Big((2k^2-1)\E(k)+(1-k^2)\K(k)\Big),
\end{eqnarray*}
In addition, since $\omega$ is a function of $k$, by \eqref{5.0} and the chain rule we have
\begin{equation}\begin{array}{llllll}\label{dc2}
\mathsf{d}''(c) &=& \displaystyle-\int_0^{L}(\varphi'(x))^2dx+2(1-\omega)\frac{d}{dk}\int_0^{L}(\varphi'(x))^2dx \, \left(\frac{d\omega}{dk}\right)^{-1}  \\
& & \\
& = &\displaystyle \frac{-32\K(k)}{3(2k^2-1)L}\Big((2k^2-1)\E(k)+(1-k^2)\K(k)\Big)+\\
& & \\
& & \displaystyle+\frac{64(1-\omega)}{3L}\frac{d}{dk}\left(\frac{\K(k)}{(2k^2-1)}\Big((2k^2-1)\E(k)+(1-k^2)\K(k)\Big)\right)\left(\frac{d\omega}{dk}\right)^{-1}.\\
\end{array}\end{equation}
In $(\ref{dc2})$, we denote
$$
m(k):=\frac{-32\K(k)}{3(2k^2-1)L} \left( (2k^2-1)\E(k)+(1-k^2)\K(k)\right)
$$
and 
$$
n(k):=\frac{64(1-\omega)}{3L}\frac{d}{dk}\left(\frac{K(k)}{(2k^2-1)}\Big((2k^2-1)E(k)+(-k^2+1)K(k)\Big)\right)\left(\frac{d\omega}{dk}\right)^{-1}.
$$
After some algebraic calculations (see Figure $\ref{fig1}$), we see that  $m(k)$ and $n(k)$ satisfies
$$
m(k)+n(k)<0,\; \text{for all} \;  k \in  \left(\frac{1}{\sqrt{2}},1\right)
$$
so that $\mathsf{d}''(c)<0,$ for $c \in [0,1)$.

\begin{figure}[h]\label{fig1}
\centering 
\includegraphics[width=6cm]{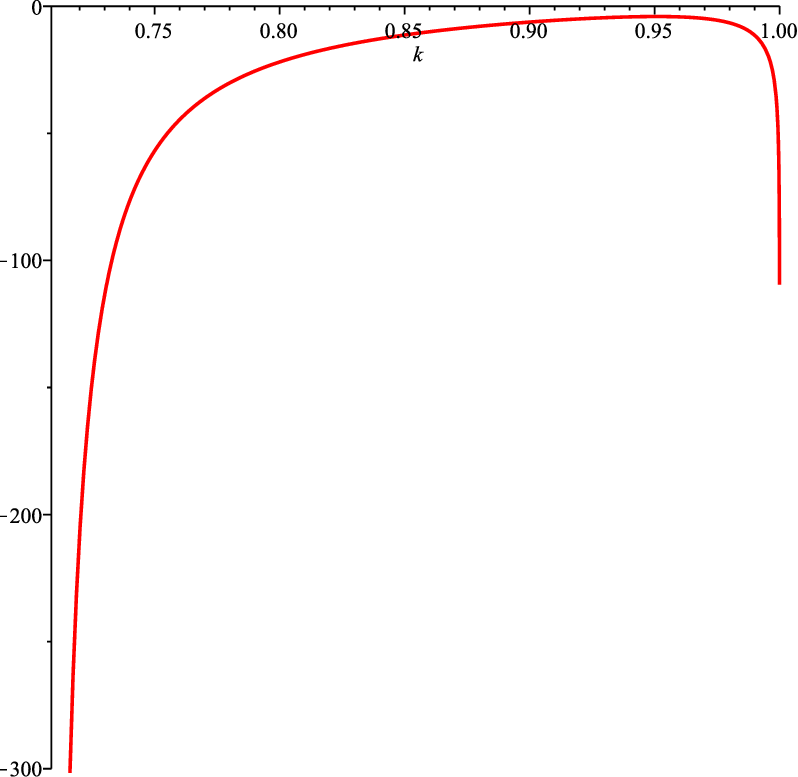} 
\caption{Graphic of $\mathsf{d}''$ as a function of $k\in\left(\frac{1}{\sqrt2},1\right)$.}
\end{figure}

Analysis above gives the following orbital instability result.

\begin{theorem}[Orbital instability of the cnoidal waves for the KG equation]\label{KGmaininstabilityresult}
Let $L>0$ be fixed. If $c \in [0,c(k_1))$, where $k_1$ is given in  (\ref{k1}) and $\varphi$ is the cnoidal solution given in Proposition \ref{cnoidalsol}, then the periodic wave $\Phi$ is orbitally unstable in $X$ in the sense of Definition $\ref{stadef}$. 
\end{theorem}
\begin{proof}
By Proposition \ref{teoeigenKGconst}, there exists only one negative eigenvalue for the operator $\mathcal{L}_{\Pi}$ in \eqref{opconstrained2} and it results to be simple. In addition, zero is a simple eigenvalue associated to the eigenfunction $(\varphi', c\varphi'')$. The result then follows from \cite[Theorem 4.7]{grillakis1} and the fact that $\mathsf{d}''(c)<0$ for all $c\in [0,c(k_1))\subset [0,c^{\ast})$.
\end{proof}

\begin{remark}
 In view of the results obtained in Proposition $\ref{teoeigenKGconst}$ and the considerations discussed in Remark $\ref{wellposednessRemark}$, it is reasonable to consider equation \eqref{KF2} restricted to the space consisting of solutions with the zero mean property. Moreover, we emphasize that our main result on orbital instability for equation \eqref{KF2}, presented in Theorem $\ref{KGmaininstabilityresult}$, strongly relies on the fact that orbital instability in a subspace implies orbital instability in the whole space. 
\end{remark}

\begin{remark}
	By symmetry and the fact that $\omega=1-c^2>0$, we can consider $c\in (-c(k_1),0)\subset (-c^{\ast},0)$ in Theorem $\ref{KGmaininstabilityresult}$ without further problems.
\end{remark}

\section{Orbital stability of the cnoidal waves for the NLS equation}\label{section6}

The goal of this section is to establish a result of orbital stability, based on the Lyapunov functional  arguments contained in \cite{NP2015} (see also \cite{stuart}), for the periodic standing waves solutions for the NLS equation of the form \eqref{PW2},
where $\omega \in (0,\omega^*)$ is given in Proposition \ref{propn1} and $\varphi:=\varphi_{\omega}$ given by Proposition \ref{cnoidalsol2}.

Equation \eqref{KF1} has two basic symmetries, namely, translation and rotation. In other words, if $u = u(x,t)$ is a solution of \eqref{KF1}, so are $e^{-i\theta}u$ and $u(x-r,t)$ for any  $\theta, r \in \mathbb{R}$. By writing $U=P+iQ\equiv(P,Q)$ this is equivalent to saying that \eqref{KF1} is invariant by the transformations
\begin{equation*}
	T_1(\theta)U := \left(
	\begin{array}{cc}
		\cos{\theta} & \sin{\theta} \\
		- \sin{\theta} & \cos{\theta}
	\end{array}
	\right) \left(
	\begin{array}{c}
		P \\ Q
	\end{array}
	\right)
\end{equation*} 
and
\begin{equation*}
	T_2(r)U := \left(
	\begin{array}{c}
		P(\cdot - r, \cdot) \\
		Q(\cdot - r, \cdot)
	\end{array}
	\right).
\end{equation*}
The actions $T_1$ and $T_2$ define unitary groups in $\mathbb{H}^1_{per,m}$ with infinitesimal generators given by
\begin{equation*}
	T_1'(0)U := \left(
	\begin{array}{cc}
		0 & 1 \\
		-1 & 0
	\end{array}
	\right) \left(
	\begin{array}{c}
		P \\
		Q
	\end{array}
	\right) = J \left(
	\begin{array}{c}
		P \\
		Q
	\end{array}
	\right)
\end{equation*}
and
\begin{equation*}
	T_2'(0)U := -\partial_x \left(
	\begin{array}{c}
		P \\
		Q
	\end{array}
	\right).
\end{equation*}

A standing wave solution  as in (\ref{PW2}) is then of the form
\begin{equation*}
	U(x,t) = \left(
	\begin{array}{c}
		\varphi(x) \cos(\omega t) \\
		\varphi(x) \sin(\omega t)
	\end{array}
	\right).
\end{equation*}  
Taking into account that the NLS equation is invariant by the actions of  $T_1$ and $T_2$, we define the orbit generated by 
$
\Phi = (\varphi,0)
$
as
\begin{equation*}
	\Omega_\Phi = \big\{T_1(\theta)T_2(r)\Phi; \theta, r \in \R \big\} \equiv \left\{ \left(
	\begin{array}{cc}
		\cos{\theta} & \sin{\theta} \\
		-\sin{\theta} & \cos{\theta}
	\end{array}
\right) \left(
	\begin{array}{c}
	\varphi(\cdot - r) \\
	0
	\end{array}
\right) \; ; \; \theta, r \in \R \right\}.
\end{equation*}

In $\mathbb{H}^1_{per,m}$, we introduce the pseudometric $d$ by
\begin{equation*}
	d(f,g):= \inf \{ \|f - T_1(\theta)T_2(r)g\|_{\mathbb{H}^1_{per}}\; ; \; \, \theta, r \in \R\}.
\end{equation*}
In particular, the distance between $f$ and $g$ is the distance between $f$ and the orbit generated by $g$ under the action of rotation and translation. In particular, 
$
	d(f,\Phi) = d(f,\Omega_\Phi).
$

We present now our notion of orbital stability.

\begin{definition}\label{defistabilitystanding}
	Let $\Theta(x,t) = (\varphi(x)\cos(\omega t), \varphi(x) \sin(\omega t))$ be a standing wave for \eqref{hamiltonian-schrodinger}, where $\omega \in (0,+\infty)$. We say that $\Theta$ is orbitally stable in $\mathbb{H}^1_{per,m}$ provided that, given $\varepsilon > 0$, there exists $\delta > 0$ with the following property: if $U_0 \in \mathbb{H}^1_{per,m}$ satisfies $\|U_0 - \Phi\|_{\mathbb{H}^1_{per}} < \delta$, then the weak solution $U(t)$ for \eqref{hamiltonian-schrodinger} with initial data $U_0$ exists for all $t \geq 0$ in $\mathbb{H}^1_{per,m}$ and satisfies
	\begin{equation*}
		d(U(t), \Omega_\Phi) < \varepsilon, \; \text{for all} \; t \geq 0.
	\end{equation*}

Otherwise, we say that $\Theta$ is orbitally unstable in $\mathbb{H}^1_{per,m}$.
\end{definition}

\subsection{Global well-posedness}

The first step to obtain the orbital stability  is to show the global existence of weak solutions  in $H_{per,m}^1$, which we give next.

\begin{theorem}\label{wellposednessSchrodinger}
	The initial-value problem associated with \eqref{KF1} is globally well-posed in $H^1_{per,m}$. More precisely, for any $u_0 \in H^1_{per,m}$ and any $T>0$, there exists a unique global (weak) solution of \eqref{KF1} in the sense that
	\begin{equation}\label{cond1}
		u \in L^\infty(0,T; H^1_{per,m}), \, \, u_t \in L^\infty(0,T;H^{-1}_{per,m}),
	\end{equation}
and $u$ satisfies, 
\begin{equation}\label{weakeq}
	i\langle u_t(\cdot,t), \phi\rangle_{H^{-1}_{per,m},H_{per,m}^1} - \int_{0}^{L} u_x(\cdot,t)\overline{\phi_x} \;dx + \int_{0}^{L} u(\cdot,t)|u(\cdot,t)|^2\overline{\phi}\;dx = 0, \, \mbox{ a.e. } t \in [0,T],
\end{equation}
for all $\phi \in H^1_{per,m}$. Furthermore, $u$  satisfies $u(\cdot,0) = u_0$ and there exist constants $c_1, c_2 >0$ such that
\begin{equation}\label{conserv}
	\mathcal{E}(u(\cdot,t)) \equiv c_1 \, \mbox{ and } \, \mathcal{F}(u(\cdot,t)) \equiv c_2, \, \, \mbox{ a.e. } t \in [0,T],
\end{equation}
where $\mathcal{E}$ and $\mathcal{F}$ are defined in \eqref{quantEF1} and \eqref{quantEF2}. Equalities in $(\ref{conserv})$ allow to conclude from $(\ref{cond1})$ that 
\begin{equation}\label{cond12}
	u \in C([0,T]; H^1_{per,m}), \, \, u_t \in C([0,T];H^{-1}_{per,m}).
\end{equation}

\end{theorem}
\begin{proof}
To prove this result, we use Galerkin's method. Since the method is quite well-known we give only the main steps  (for some additional details, see the Appendix). Let $\{w_j\}_{j \in \mathbb{Z}\backslash\{0\}}$ be the sequence given by
\begin{equation*}
w_j(x)=\frac{1}{\sqrt{L}}e^{\frac{2\pi i}{L}jx},
\end{equation*}
which is a complete orthonormal set in $L^2_{per,m}$ and orthogonal in $H^1_{per,m}$.

Let $T>0$ be arbitrary but fixed. For each positive integer $n$, let $V_n$ be the finite subspace generated by $w_{-n},w_{-n+1},\ldots,w_n$. By using the Carathéodory Theorem (\cite[Chapter 2, Theorem 1.1]{CoddingtonLevinson} and \cite[Chapter 1]{filippov}), we can show the existence of functions $g_{jn}$ such that the function
\begin{equation*}
	u_n(t) = \sum_{{{j=-n}}\atop{j\neq0}}^n g_{jn}(t) w_j\in V_n,
\end{equation*}
satisfies, for each $j \in \{-n, \ldots,-1,1,\ldots n\}$, the following approximate problem
\begin{equation}\label{approblem}
	\left\{
	\begin{array}{llll}
		\displaystyle i\langle u_{n,t}(t), w_j\rangle_{H_{per,m}^{-1}, H_{per,m}^1} - \int_0^Lu_{n,x}(\cdot,t)\overline{w_{j,x}}dx + \int_0^Lu_n(\cdot,t)|u_n(\cdot, t)|^2 \overline{w_j}dx = 0,  \\ \\
		u_n(0) = u_{0n} = \displaystyle{ \sum_{{{j=-n}}\atop{j\neq0}}^n (u_{0},w_j)_{{H}^1_{per}} w_j} \longrightarrow u_0 \mbox{ in } H^1_{per,m}, \; \text{as} \; n \to \infty
		\end{array}
	\right.
\end{equation}
for all $t\in [0,T]$.  Note that a priori estimates allow us to define functions $u_{n}(t)$ in the interval $[0,T]$. Since $w_j \in H^1_{per,m}$ for all $j\in \mathbb{Z}\backslash\{0\}$ and $H^1_{per,m}$ is a closed subspace of $H^1_{per}$, we obtain all standard steps concerning Galerkin's method and the proof of the theorem follows.
\end{proof}

\subsection{Positivity of the operator $\boldsymbol{\mathcal{G}''(\Phi)}$} 

In this section our goal is to study the positivity of the operator $\mathcal{G}''(\Phi)$ given in \eqref{matrixop2}. Before that, we need to prove some positivity properties for the operators $\mathcal{L}_2$ and $\mathcal{L}_3$ in \eqref{L2L3}.

\begin{lemma}\label{lema-L3}
	There exists $\delta_2 > 0$ such that
	\begin{equation}\label{delta2}
		(\mathcal{L}_{3\Pi} \, v, v)_{L^2_{per}} \geq \delta_2 \|v\|_{L^2_{per}}^2,
	\end{equation}
for all $v \in H^2_{per,m}$ satisfying $(v,\varphi)_{L^2_{per}} = 0$.
\end{lemma}
\begin{proof}
	Since $L^2_{per,m}$ is a Hilbert space and $\varphi \in L^2_{per,m}$, we have the decomposition
$
		L^2_{per,m} = N \oplus M,
$
where $N:=[\varphi]$ and $M:= [\varphi]^\perp$. From Theorem 6.17 in \cite[page 178]{katobook}, we obtain   
\begin{equation*}
	\sigma\left(\mathcal{L}_{3\Pi}\right) = \sigma\left(\mathcal{L}_{3\Pi, N}\right) \cup \sigma\left(\mathcal{L}_{3\Pi,M}\right),
\end{equation*}
where $\mathcal{L}_{3\Pi, N}$ and $\mathcal{L}_{3\Pi, M}$ denote the operator $\mathcal{L}_{3 \Pi}$ restricted to $N$ and $M$, respectively, and $\sigma(\mathcal{A})$ denotes the spectrum of the linear operator $\mathcal{A}$.
Since $\text{Ker}(\mathcal{L}_{3 \Pi})=[\varphi]$, we obtain
\begin{equation*}
\sigma\left(\mathcal{L}_{3\Pi,M}\right) = \sigma\left(\mathcal{L}_{3\Pi}\right) \setminus \{0\},
\end{equation*}
that is, the spectrum of $\mathcal{L}_{3\Pi,M}$ is bounded from below. The arguments in \cite[page 278]{katobook}, imply that $\mathcal{L}_{3\Pi}$ is also bounded from below, so that there exists $\delta_2 > 0$ satisfying
\begin{equation*}
	(\mathcal{L}_{3\Pi}\, v, v)_{L^2_{per}} \geq \delta_2 \|v\|_{L^2_{per}}, \, \, \text{for all} \, v \in M \cap H^2_{per,m},
\end{equation*}
Therefore, for all $v \in H^2_{per,m}$ such that $(v, \varphi)_{L^2_{per}} = 0$ we have that \eqref{delta2} holds.
\end{proof}

Before proving the next result, we recall the convexity of the function $\mathsf{d}$ defined by $\mathsf{d}(\omega)=\mathcal{G}(\Phi)=\mathcal{E}(\Phi)+\omega\mathcal{F}(\Phi)$.

\begin{proposition}\label{convexityd2}
	Let $L>0$ be fixed. For the cnoidal wave $\varphi$ given by Proposition \ref{cnoidalsol2}, we have $\mathsf{d}''(\omega)>0$ for all $\omega \in (0+\infty)$.
\end{proposition}
\begin{proof}  See \cite[page 23]{angulo}.
\end{proof}
\indent Next result gives us that the quadratic form associated to the operator $\mathcal{L}_{2\Pi}$ is non-negative when restricted to $[\varphi]^\perp$.
\begin{lemma}\label{lema-gamma-zero}
	Let $\gamma$ be defined as
	\begin{equation}\label{gamma}
		\gamma := \inf \left\{ (\mathcal{L}_{2\Pi}\, v,v)_{L^2_{per}} : v \in H^2_{per,m}, \, (v, \varphi)_{L^2_{per}} = 0, \, \|v\|_{L^2_{per}} = 1 \right\}.
	\end{equation}
Then $\gamma = 0$.
\end{lemma}
\begin{proof}
Let us define
	\begin{equation}\label{gamma0}
	\gamma_0 := \inf \left\{ \langle \mathcal{L}_{2\Pi} v,v\rangle : v \in H^1_{per,m}, \, (v, \varphi)_{L^2_{per}} = 0, \, \|v\|_{L^2_{per}} = 1 \right\},
	\end{equation}
where $\langle \cdot, \cdot \rangle := \langle \cdot, \cdot \rangle_{H^{-1}_{per,m},H^1_{per,m}}$.
Let us first prove that $\gamma_0 = 0$ and the infimum in \eqref{gamma0} is achieved. In fact, since $\varphi$ is periodic, we have
\begin{align*}
	\langle \mathcal{L}_{2\Pi} v, v \rangle & = \|v'\|_{L^2_{per}} + \omega - 3 (\varphi^2, v^2)_{L^2_{per}} \geq -3 \, (\varphi^2, v^2)_{L^2_{per}}  \geq -3 \, \|\varphi^2\|_{L^{\infty}_{per}},
\end{align*}
for all $v \in H^1_{per,m}$ such that $\|v\|_{L^2_{per}} = 1$. This implies that $\gamma_0$ is finite. Furthermore, since
$(\varphi, \varphi')_{L^2_{per}} = 0$ and  $\mathcal{L}_{2\Pi}\varphi' = 0$
we get $\gamma_0 \leq 0.$ 

Now, we prove that the infimum in \eqref{gamma0} is achieved. Indeed, let 
$
\{u_j\}_{j \in \mathbb{N}} \subset H^1_{per,m}
$
be a sequence such that
\begin{equation}\label{17021}
 \|u_j\|_{L^2_{per}} = 1, \, \, \, (u_j, \varphi)_{L^2_{per}} = 0, \;\text{for all} \; j \in \mathbb{N}
\, \mbox{ and } \, 	\langle\mathcal{L}_{2\Pi} u_j, u_j\rangle \longrightarrow \gamma_0, \; \text{as} \; j \rightarrow \infty.
\end{equation}
Since the sequence $\left\{ \langle\mathcal{L}_{2\Pi} u_j, u_j\rangle \right\}_{j \in \mathbb{N}}$ is bounded in $\mathbb{R}$, there exists $C>0$ so that
\begin{equation}\label{est1}
	\langle\mathcal{L}_{2\Pi} u_j, u_j\rangle \leq C \; \text{for all} \; j \in \mathbb{N}.
\end{equation}
Now, since $\varphi$ is bounded we obtain from $(\ref{est1})$ the existence of $C_1>0$ such that
\begin{equation*}
	\omega\|u_j\|_{L^2_{per}}^2 + \|u_j'\|_{L^2_{per}}^2 = \langle\mathcal{L}_{2\Pi} u_j, u_j\rangle + 3(\varphi^2,u_j^2)_{L^2_{per}} \leq C_1, \; \text{for all} \; j \in \mathbb{N},
\end{equation*}
which gives that $\{u_j\}_{j \in \mathbb{N}}$ is bounded in $H^1_{per,m}$. Hence, there exists a subsequence, still denoted by $\{u_j\}_{j \in \mathbb{N}}$ and $u \in H^1_{per,m}$ such that
$	u_j \rightharpoonup u \, \mbox{ in } \, H^1_{per,m}.$ By \eqref{17021} and the compact embedding $H^1_{per,m} \hookrightarrow L^2_{per,m}$, we have
$
	\|u\|_{L^2_{per}} = 1 \, \mbox{ and } \, (u,\varphi)_{L^2_{per}} = 0.
$
In addition,  Fatou's Lemma yields
\begin{equation*}
	\gamma_0 \leq \langle \mathcal{L}_{2\Pi} \, u, u\rangle \leq \liminf_{j\in \mathbb{N}} \langle \mathcal{L}_{2\Pi}\,  u_j, u_j\rangle = \gamma_0.
\end{equation*}
Last inequality allows to deduce that the infimum  in (\ref{gamma0}) is achieved at the function $u$.

It remains to prove that $\gamma_0 \geq 0$. Let
\begin{equation*}
	\Upsilon_0 := \left\{ v \in H^1_{per,m}, \, (v, \varphi)_{L^2_{per}} = 0, \, \|v\|_{L^2_{per}} = 1 \right\}.
\end{equation*}
It is then suffices to show  
\begin{equation*}
\langle \mathcal{L}_{2\Pi}v,v\rangle \geq 0, \; \text{for all} \; v \in \Upsilon_0.
\end{equation*}
For this, let us start by noting that if $\chi = -\dfrac{\partial \varphi}{\partial \omega}$ then $\mathcal{L}_{2}\chi = \varphi$, and
\begin{equation}\label{3a}
\langle \mathcal{L}_{2\Pi}\chi,\chi \rangle  = \langle \mathcal{L}_{2}\chi,\chi \rangle   = \left\langle \mathcal{L}_{2}\dfrac{\partial \varphi}{\partial \omega},\dfrac{\partial \varphi}{\partial \omega} \right\rangle = \left( \mathcal{L}_2 \dfrac{\partial \varphi}{\partial \omega}, \dfrac{\partial \varphi}{\partial \omega} \right)_{L^2_{per}} = - \left( \varphi, \dfrac{\partial \varphi}{\partial \omega} \right)_{L^2_{per}} = -\mathsf{d}''(\omega).
\end{equation}
By Proposition \ref{convexityd2}, we have
\begin{equation}\label{17-02-1}
	\langle \mathcal{L}_{2\Pi} \chi, \chi \rangle < 0.
\end{equation}

Now, recall that $\text{Ker}(\mathcal{L}_{2\Pi}) = [\varphi']$ and $\text{n}(\mathcal{L}_{2\Pi}) = 1.$ Then,  there exists $-\lambda^2 < 0$ and $\Psi \in D(\mathcal{L}_{2\Pi})$ such that $\|\Psi\|_{L^2_{per}} = 1$ and $\mathcal{L}_{2\Pi} \Psi = -\lambda^2 \Psi.$ In addition we may write
\begin{equation*}
	L^2_{per,m} = [\Psi] \oplus [\varphi'] \oplus P.
\end{equation*}
where $P \subset L^2_{per,m}$ is such that 
\begin{equation*}
	\langle \mathcal{L}_{2\Pi} u, u\rangle \geq \delta \|u\|_{L^2_{per}}^2,\; \text{for all} \; u  \in \tilde{P}:= D(\mathcal{L}_{2\Pi}) \cap P.
\end{equation*}
Next, since $\chi\in H_{per,m}^1$, we infer
\begin{equation*}
	\chi = a_0 \Psi + b_0 \varphi' + p_0,
\end{equation*}
where $a_0, b_0 \in \mathbb{R}$ and $p_0 \in \tilde{P}$. By \eqref{17-02-1} and the fact that $\Psi$ is orthogonal to the elements in the positive subspace $\tilde{P}$, it follows that
\begin{equation}\label{17-02-4}\begin{array}{llll}
 0 &>& \langle \mathcal{L}_{2\Pi} ( a_0 \Psi + b_0 \varphi'+p_0), a_0 \Psi + b_0 \varphi'+p_0\rangle \\\\
	 &=&a_0^2\langle\mathcal{L}_{2\Pi}\Psi,\Psi  \rangle+2a_0\langle\mathcal{L}_{2\Pi}\Psi,p_0\rangle +\langle\mathcal{L}_{2\Pi}p_0,p_0\rangle\\\\
	 &=&a_0^2\langle\mathcal{L}_{2\Pi}\Psi,\Psi  \rangle-2a_0\lambda^2\langle\Psi,p_0\rangle +\langle\mathcal{L}_{2\Pi}p_0,p_0\rangle\\\\
	 &=&- a_0^2\lambda^2+\langle \mathcal{L}_{2\Pi} p_0, p_0 \rangle.
\end{array}\end{equation}

Let  $v \in \Upsilon_0$ be arbitrary. First, if $v = \beta \varphi'$, for some $\beta\neq0$ then clearly
\begin{equation}\label{17-02-5}
	\langle \mathcal{L}_{2\Pi} \tilde{v}, \tilde{v} \rangle = 0 ,\qquad \tilde{v} = \frac{v}{\|v\|_{L^2_{per}}}\in \Upsilon_0.
	\end{equation}
On the other hand, if $v \notin [\varphi']$, there exist $a_1, b_1 \in \R$ and $p_1 \in P$ such that
\begin{equation*}
	v = a_1 \Psi + b_1 \varphi' + p_1.
\end{equation*}
Thus,
\begin{equation}\label{17-02-6}\begin{array}{llll}
0 = -(\varphi, v)_{L^2_{per}} & = \left( \mathcal{L}_{2\Pi}\chi, v\right)_{L^2_{per}} \\
& = \langle \mathcal{L}_{2\Pi}\chi, v \rangle \\
& = \langle -a_0 \lambda^2 \Psi + \mathcal{L}_{2\Pi}p_0, a_1 \Psi + b_1 \varphi' + p_1 \rangle \\
& = -a_0 a_1 \lambda^2 + \langle \mathcal{L}_{2\Pi}p_0, p_1 \rangle.
\end{array}\end{equation}
Using the fact that $\xi(f,g) = \langle \mathcal{L}_{2\Pi}f, g \rangle$ defines a non-negative sesquilinear form on $\tilde{P}$, we obtain the Cauchy-Schwarz inequality:
\begin{equation}\label{17-02-7}
	|\xi (f,g)|^2 \leq \langle \mathcal{L}_{2\Pi}f,f\rangle \, \langle \mathcal{L}_{2\Pi}g,g\rangle, \; \text{for all} \; f, g \in \tilde{P}.
\end{equation}

Next, we see that
\begin{equation*}
\langle \mathcal{L}_{2\Pi} v, v\rangle = \langle \mathcal{L}_{2\Pi}(a_1 \Psi + b_1 \varphi' + p_1), a_1 \Psi + b_1\varphi' + p_1 \rangle = -a_1^2\lambda^2 + \langle \mathcal{L}_{2\Pi}p_1, p_1\rangle.
\end{equation*}
Additionally, using \eqref{17-02-7}, we obtain
\begin{equation*}
\langle \mathcal{L}_{2\Pi} v, v\rangle \geq -a_1^2 \lambda^2 + \dfrac{|\langle \mathcal{L}_{2\Pi}p_0, p_1\rangle|^2}{\langle \mathcal{L}_{2\Pi}p_0, p_0\rangle}.
\end{equation*}
By \eqref{17-02-4} and the equality given by \eqref{17-02-6}, we conclude that
\begin{equation*}
\langle \mathcal{L}_{2\Pi} v, v\rangle > -a_1^2 \lambda^2 + \dfrac{a_0^2 a_1^2 \lambda^4}{a_0^2 \lambda^2} = 0,
\end{equation*}
that is,
\begin{equation}\label{17-02-8}
\langle \mathcal{L}_{2\Pi} v, v\rangle > 0,\; \text{for all} \; v \in \Upsilon_0\ \mbox{with}\ v \notin [\varphi'].
\end{equation}
By \eqref{17-02-5} and \eqref{17-02-8} we obtain $\gamma_0 \geq 0$.

In order to complete the proof of the lemma it suffices to show that $\gamma = \gamma_0$. To this end, take $u \in \Upsilon_0$ satisfying
$
	\langle \mathcal{L}_{2\Pi} u,u \rangle = \gamma_0.
$
Since the minimum of the constrained functional $\langle\mathcal{L}_{2\Pi}u,u\rangle$ over $\Upsilon_0$ is attained, we obtain, by the Lagrange multiplier theorem, that there exist $\eta, \zeta \in \R$ such that
\begin{equation}\label{est2}
	\mathcal{L}_{2\Pi}u = \eta u + \zeta \varphi.
\end{equation}
We see by \eqref{est2} that $u \in H^2_{per,m}$, which allows to conclude that $\gamma = \gamma_0 = 0$ and the proof of the lemma is completed.
\end{proof}

\begin{lemma}\label{lema-xi-positivo}
	Let $\xi$ be defined as
	\begin{equation}\label{xi}
		\xi := \inf \left\{ (\mathcal{L}_{2\Pi}v,v)_{L^2_{per}} : v \in H^2_{per,m}, \, (v, \varphi)_{L^2_{per}} = (v, \varphi')_{L^2_{per}} = 0, \, \|v\|_{L^2_{per}} = 1 \right\}.
	\end{equation}
	Then $\xi > 0$.
\end{lemma}

\begin{proof}
From Lemma \ref{lema-gamma-zero} it is clear that $\xi\geq0$. Assume by contradiction  that $\xi = 0$. Using similar arguments as in Lemma \ref{lema-gamma-zero} we may show that the infimum in \eqref{xi} is attained at a function $u \in H^2_{per,m}$ satisfying
\begin{equation}\label{lema3-1}
	\|u\|_{L^2_{per}} = 1, \, \, (u,\varphi)_{L^2_{per}} = (u,\varphi')_{L^2_{per}} = 0\ \mbox{and}\  (\mathcal{L}_{2\Pi}u,u)_{L^2_{per}} = \xi = 0.
	\end{equation}
By the Lagrange multiplier theorem, there exist $\lambda, \theta,\nu \in \R$ such that
\begin{equation}\label{lema3-2}
	\mathcal{L}_{2\Pi}u = \lambda u + \theta \varphi + \nu \varphi'.
\end{equation}
From \eqref{lema3-1} and \eqref{lema3-2}, we clearly have $\lambda = 0$. Also, taking the inner product in \eqref{lema3-2} with $\varphi'$ and taking into account that $\mathcal{L}_{2\Pi}$ is self-adjoint we deduce that
 $\nu = 0$ and consequently $\mathcal{L}_{2\Pi}u = \theta \varphi$. Recalling that $\mathcal{L}_{2\Pi}\chi = \varphi$, where $\chi = -\dfrac{\partial \varphi}{\partial \omega}$, we have
$
\mathcal{L}_{2\Pi}(\theta \chi) = \theta \varphi = \mathcal{L}_{2\Pi}u
$
which implies 
$
\mathcal{L}_{2\Pi} (u - \theta \chi) = 0. 
$
Since $\text{Ker}(\mathcal{L}_{2\Pi}) = [\varphi']$, there exists $\sigma \in \R$ such that $u - \theta \chi = \sigma \varphi'$. On the other hand, using that $(\chi, \varphi)_{L^2_{per}} = -\mathsf{d}''(\omega) < 0$ (see \eqref{3a}) and $(\varphi',\varphi)_{L^2_{per}} = 0$, we obtain by \eqref{lema3-1} that
$
- \theta (\chi,\varphi)_{L^2_{per}}=(\sigma\varphi'-u,\varphi)_{L^2_{per}} = 0,
$
and hence $\theta = 0$. Thus, we conclude  that $u = \sigma \varphi'$ for some $\sigma\neq0$. But this is a contradiction because from \eqref{lema3-1},
$
	0 = (u,\varphi')_{L^2_{per}} = \sigma\|\varphi'\|_{L^2_{per}}^2\neq0.
$
\end{proof}

\begin{corollary}\label{corolario1}
	Let $\mathcal{L}_{\Pi}: \mathbb{H}^2_{per,m} \subset \mathbb{L}^2_{per,m} \rightarrow \mathbb{L}^2_{per,m}$ be defined by
	\begin{equation*}
		\mathcal{L}_{\Pi} := \left(
		\begin{array}{cc}
			\mathcal{L}_{2\Pi} & 0 \\
			0 & \mathcal{L}_{3\Pi}
		\end{array}
		\right),
	\end{equation*}
	where $\mathcal{L}_{2\Pi} $ and $ \mathcal{L}_{3\Pi}$ are given in \eqref{opconstrained123} and \eqref{opconstrained1234}, respectively. If $f=(u,v) \in \mathbb{H}^2_{per,m}$ satisfies
	\begin{equation}\label{18-02-1}
		(v,\varphi)_{L^2_{per}} = (u,\varphi)_{L^2_{per}} = (u, \varphi')_{L^2_{per}} = 0
	\end{equation}
	then there exists $\delta > 0$ such that
	\begin{equation*}
		(\mathcal{L}_{\Pi}f, f)_{\mathbb{L}^2_{per}} \geq \delta \|f\|^2_{\mathbb{L}^2_{per}}.
	\end{equation*}
\end{corollary}

\begin{proof}
	Let $f=(u,v) \in \mathbb{H}^2_{per,m}$ be satisfying \eqref{18-02-1}. From Lemmas \ref{lema-L3} and \ref{lema-xi-positivo}, there exist $\delta_1, \delta_2 > 0$ such that 
	\begin{equation*}
		(\mathcal{L}_{2\Pi}\,u, u)_{L^2_{per}} \geq \delta_1 \|u\|^2_{L^2_{per}} \, \, \mbox{ and } \, \, (\mathcal{L}_{3\Pi}\,v, v)_{L^2_{per}} \geq \delta_2 \|v\|^2_{L^2_{per}}.
	\end{equation*}
Thus,
\begin{equation*}
	(\mathcal{L}_{\Pi}f,f)_{\mathbb{L}^2_{per}} = (\mathcal{L}_{2\Pi}\,u, u)_{L^2_{per}} + (\mathcal{L}_{3\Pi}\,v, v)_{L^2_{per}}\geq \delta_1 \|u\|^2_{L^2_{per}}+\delta_2 \|v\|^2_{L^2_{per}}.
\end{equation*}
Defining $\delta:= \min\{\delta_1, \delta_2\} > 0$, the last inequality gives the desired.
\end{proof}

\begin{lemma}\label{positividade-L}
	There exist positive constants $\alpha_1$ and $\alpha_2 $ such that
	\begin{equation*}
		(\mathcal{L}_{\Pi}f,f)_{\mathbb{L}^2_{per}} \geq \alpha_1 \|f\|^2_{\mathbb{H}^1_{per}} - \alpha_2\|f\|^2_{\mathbb{L}^2_{per}},
	\end{equation*}
for all $f = (u,v) \in \mathbb{H}^2_{per,m}$.
\end{lemma}

\begin{proof}
	Let $f=(u,v) \in \mathbb{H}^2_{per,m}$ be fixed. Since $\mathcal{L}_{2}$ and $\mathcal{L}_{3}$ are second order differential operators, we obtain from G$\mathring{\text{a}}$rding's inequality (see \cite[page 175]{yosida}) that
	\begin{equation*}
		(\mathcal{L}_{2\Pi}\,u,u)_{L^2_{per}} = (\mathcal{L}_2\, u,u)_{L^2_{per}} \geq \beta_1 \|u\|_{H^1_{per}}^2 - \beta_2 \|u\|_{L^2_{per}}^2
	\end{equation*}
and
\begin{equation*}
	(\mathcal{L}_{3\Pi}\,v,v)_{L^2_{per}} = (\mathcal{L}_3 \,v,v)_{L^2_{per}} \geq \beta_3 \|v\|_{H^1_{per}}^2 - \beta_4 \|v\|_{L^2_{per}}^2.
\end{equation*}
By setting
$
	\alpha_1 = \min\{\beta_1,\beta_3\} >0$ and $\alpha_2= \max\{\beta_2,\beta_4\}>0,
$ we obtain
\begin{equation*}
	(\mathcal{L}_{\Pi}f,f)_{\mathbb{L}^2_{per}} = (\mathcal{L}_{2\Pi}\,u,u)_{L^2_{per}} + (\mathcal{L}_{3\Pi}\,v,v)_{L^2_{per}} \geq \alpha_1 \|f\|^2_{\mathbb{H}^1_{per}} - \alpha_2\|f\|^2_{\mathbb{L}^2_{per}},
\end{equation*}
which completes the proof of the lemma.
\end{proof}

\begin{remark}\label{remark1}
	Let $\omega>0$ be fixed. Recall that $\mathcal{G}: \mathbb{H}^1_{per} \rightarrow \R$ is given by 
$
		\mathcal{G} = \mathcal{E} + \omega\mathcal{F},
$
and
$
	\mathcal{G}''(\Phi): \mathbb{H}^1_{per} \rightarrow \mathbb{H}_{per}^{-1}.
$
Let  $\mathsf{b}: D(\mathsf{b})= \mathbb{H}^3_{per} \subset \mathbb{L}^2_{per} \rightarrow \R$ be the quadratic form associated to the operator $\mathcal{G}''(\Phi)$, that is, 
\begin{equation*}
	\mathsf{b}(v) = \langle \mathcal{G}''(\Phi) v, v \rangle_{\mathbb{H}^{-1}_{per}, \mathbb{H}^{1}_{per}} = (\mathcal{L} v,v)_{\mathbb{L}^2_{per}},\; \text{for all}  \; v \in D(\mathsf{b}).
\end{equation*}
Note that $\mathsf{b}$ is densely defined and, from Lemma \ref{positividade-L}, $\mathsf{b}$ is closed and bounded from below. Thus, (see, for instance, \cite[Theorem 2.6 and page 322]{katobook}) we obtain that $\mathcal{L}$ is the unique self-adjoint linear operator such that
\begin{equation*}
	\langle \mathcal{G}''(\Phi) v, z \rangle_{\mathbb{H}^{-1}_{per}, \mathbb{H}^1_{per}} = (\mathcal{L}v, z)_{\mathbb{L}^2_{per}} = (\mathcal{L}_{\Pi}v, z)_{\mathbb{L}^2_{per}},
\end{equation*}
for all $v \in \mathbb{H}^3_{per,m}$ and $z \in \mathbb{H}^1_{per,m}$.
In this case, if $\mathcal{I}: \mathbb{H}^1_{per} \rightarrow \mathbb{H}^{-1}_{per}$ is the natural injection of $\mathbb{H}^{1}_{per}$ into $\mathbb{H}^{-1}_{per}$ with respect to the inner product in ${\mathbb{L}^2}$, that is,
\begin{equation*}
	\langle \mathcal{I}u,v \rangle_{\mathbb{H}^{-1},\mathbb{H}^1} = (u,v)_{\mathbb{L}^2_{per}},\;  \text{for all} \; u,v \in \mathbb{H}^1_{per},
\end{equation*}
then
\begin{equation*}
	\mathcal{G}''(\Phi)v = \mathcal{I} \mathcal{L} v,\; \text{for all} \; v  \in \mathbb{H}^3_{per,m}.
\end{equation*}

Since $H^1_{per,m} \hookrightarrow H^1_{per}$ and  $H^{-1}_{per} \hookrightarrow H^{-1}_{per,m}$, we can write $\mathcal{I}: \mathbb{H}^1_{per,m} \rightarrow \mathbb{H}^{-1}_{per,m}$ as
\begin{equation*}
	\langle \mathcal{I}u, v \rangle_{\mathbb{H}^{-1}_{per,m}, \mathbb{H}^1_{per,m}} = \langle \mathcal{I}u, v \rangle_{\mathbb{H}_{per}^{-1}, \mathbb{H}_{per}^1} = (u,v)_{\mathbb{L}^2_{per}},\; \text{for all} \; u, v \in \mathbb{H}^{1}_{per,m}.
\end{equation*}
\end{remark}

With the above identification in mind, we may prove the following.

\begin{lemma}\label{pos-g''-l2}
	There exists $\delta^* > 0$ such that
	\begin{equation}\label{3b}
		\langle \mathcal{G}''(\Phi) f, f \rangle_{\mathbb{H}^{-1}_{per},\mathbb{H}^1_{per}} \geq \delta^* \|f\|_{\mathbb{H}^1_{per}}^2,
	\end{equation}
for all $f = (u,v) \in \mathbb{H}^1_{per,m}$ satisfying \eqref{18-02-1}
\end{lemma}
\begin{proof}
Assume first that $f = (u,v) \in \mathbb{H}^3_{per,m}$ satisfies \eqref{18-02-1}. By Corollary \ref{corolario1}, we have
	\begin{equation*}
		\frac{1}{\delta}(\mathcal{L}_{\Pi}f,f)_{\mathbb{L}^2_{per}} \geq \|f\|_{\mathbb{L}^2_{per}}^2.
	\end{equation*}
So, by Lemma \ref{positividade-L}, we obtain
\begin{equation*}
	(\mathcal{L}_{\Pi}f,f)_{\mathbb{L}^2_{per}} \geq \alpha_1 \|f\|_{\mathbb{H}^1_{per}}^2 - \alpha_2\|f\|_{\mathbb{L}^2_{per}}^2 \geq \alpha_1 \|f\|_{\mathbb{H}^1_{per}}^2 - \frac{\alpha_2}{\delta}(\mathcal{L}_{\Pi}f,f)_{\mathbb{L}^2_{per}},
\end{equation*}
implying that
\begin{equation*}
	\left( 1 + \dfrac{\alpha_2}{\delta} \right) (\mathcal{L}_{\Pi}f,f)_{\mathbb{L}^2_{per}} \geq \alpha_1 \|f\|_{\mathbb{H}^1_{per}}^2.
\end{equation*}
Defining $\delta^* := \dfrac{\alpha_1 \delta}{\alpha_2 + \delta} > 0$ we see that \eqref{3b} holds
for any $f \in \mathbb{H}^3_{per,m}$. By density, we obtain the desired result. 
\end{proof}

\subsection{Stability of the cnoidal waves}

	Notice that Lemma \ref{pos-g''-l2} establishes the positivity of $\mathcal{G}''(\Phi)$ under the orthogonality condition \eqref{18-02-1}, which holds in $L^2_{per}$. However, this is enough to obtain the same positivity under the orthogonality condition in $H^1_{per}$ (see \cite[Lemma 4.12]{NP2015}). This allows to show that, for some positive constant $M$, the function
	$$
	V(v) = \mathcal{G}(v) - \mathcal{G}(\Phi) - M (\mathcal{F}(v) - \mathcal{F}(\Phi))^2
	$$
	is a Lyapunov function for the orbit $\Omega_\Phi$ (see \cite[Definition 4.14]{NP2015} for further details). With such a function in hand, we are able to prove the following.

\begin{theorem}[Orbital stability of the cnoidal waves for the NLS equation]\label{teo-nsl}
Let $L>0$ be fixed. If $\omega \in (0,\omega^*)$, where $\omega^*$ is given in Proposition \ref{propn1} and $\varphi=\varphi_\omega$ is the cnoidal  solution given in Proposition \ref{cnoidalsol2} then the standing wave 
$$
\Theta(x,t):=
\begin{pmatrix}
\varphi(x) \cos(\omega t) \\ 
\varphi(x) \sin(\omega t)
\end{pmatrix}
$$
is orbitally stable in  $\mathbb{H}^1_{per,m}$ in the sense of Definition $\ref{defistabilitystanding}$.
\end{theorem}
\begin{proof}
	This proof is similar to the proof of \cite[Theorem 4.17]{NP2015}, which strongly relies on the fact that $V$ (defined from $\mathcal{G}$) is a Lyapunov functional.
\end{proof}

\section*{Appendix - The Galerkin method and the existence of global weak solutions} In this appendix we present some complementary facts concerning the Galerkin approximation used in the proof of Theorem $\ref{wellposednessSchrodinger}$ (and Theorem $\ref{wellposedness1}$). Since the existence of global solutions is a crucial step in order to apply the stability method as showed in Section 6, we only present the additional facts for the equation $(\ref{KF1})$. Similar arguments can be done for  equation $(\ref{KF2})$, but it is important to mention that the orbital instability in Section 5 can be determined using the existence of local solutions. Global solutions are not necessary in that case.\\

\textit{I - A priori estimates.} Replacing $w_j$ by $u_n(t)\in V_n$ in the approximate problem $(\ref{approblem})$, we obtain after some basic calculations (well known in the case of NLS) that

\begin{equation}\label{est12}
	||u_n(t)||_{L_{per}^2}^2=||u_{0n}||_{L_{per}^2}^2,
\end{equation}
for all $t\in [0,t_n]$, where $t_n$ is the maximum time of existence obtained by Caratheodory's Theorem. On the other hand, replacing $w_j$ by $u_{n,t}(t)\in V_n$ in the same problem, we also obtain
 
 \begin{equation}\label{est13}
 	\mathcal{E}(u_n(t))=\mathcal{E}(u_{0n}),
 	\end{equation}
 for all $t\in [0,t_n]$, where $\mathcal{E}$ is defined as in $(\ref{quantEF1})$. These a priori estimates allows to extend the solution $u_n(t)$ to the whole interval $[0,T]$. Thus, using $(\ref{est12})$ and $(\ref{est13})$ combined with the Gagliardo-Nirenberg inequality for periodic domains, we obtain
 \begin{equation}\label{est14}
 ||u_{n,x}(t)||_{L_{per}^2}^2\leq 2\mathcal{E}(u_{0n})+ C||u_{n,x}(t)||_{L_{per}^2}||u_{n0}||_{L_{per}^2}^3+C||u_{n0}||_{L_{per}^2}^4.	
 	\end{equation}
 Combining $(\ref{est12})$ and $(\ref{est14})$, it follows that for each $T>0$, we have
 \begin{equation}\label{est15}
 	\{u_{n}\}\ \mbox{is bounded in}\ L^{\infty}(0,T;H_{per,m}^1),
 	\end{equation}
 for all $n\in \mathbb{Z}\backslash\{0\}$.\\
 \indent Next, using a standard argument of duality, we also obtain 
 \begin{equation}\label{est16}
 	\{u_{n,t}\}  \ \mbox{is bounded in}\ L^{\infty}(0,T;H_{per,m}^{-1}),
 \end{equation}
for all $n\in \mathbb{Z}\backslash\{0\}$.\\
\indent \textit{II - Passage to the limit.} Using $(\ref{est15})$ and $(\ref{est16})$, we obtain in particular that
\begin{equation}\label{est17}
	\{u_{n}\}\ \mbox{is bounded in}\ L^{2}(0,T;H_{per,m}^1),
\end{equation}
and
\begin{equation}\label{est18}
\{u_{n,t}\} \ \mbox{is bounded in}\ L^{2}(0,T;H_{per,m}^{-1}),
\end{equation}
for all $n\in \mathbb{Z}\backslash\{0\}$.\\
\indent Thus, up to a subsequence, the following basic convergences are obtained
\begin{eqnarray}\label{est19}\displaystyle\begin{array}{llll}
		u_n\rightharpoonup u\ \ \
		\textrm{weak in\ } \ \ L^2(0,T;H^1_{per,m}), \\
		u_{n,t}\rightharpoonup u_t\ \ \ \textrm{weak in\ } \ \
		L^2(0,T;H^{-1}_{per,m}), \\
		u_n\stackrel{\star}{\rightharpoonup}u\ \ \ \textrm{weak-$*$ in\ } \ \
		L^{\infty}(0,T;H^1_{per,m}), \\
		u_{n,t}\stackrel{\star}{\rightharpoonup}u_t\ \ \ \textrm{weak-$*$ in\ } \ \ L^{\infty}(0,T;H^{-1}_{per,m}). \\
		
	\end{array}
\end{eqnarray} 
The above convergences are enough to pass to the limit in the linear terms of the approximate problem in $(\ref{approblem})$. To handle with the nonlinear term, we need to use $(\ref{est17})$, $(\ref{est18})$ and Aubin-Lions Theorem to obtain that $u_n\rightarrow u$ strongly in $L^2(0,T;L_{per,m}^2)$. In particular and by Fubini's Theorem, we have
\begin{equation}\label{est20}
	u_n\rightarrow u\ \mbox{strongly in}\ L^2(0,T;L_{per}^2)=L^2([0,L]\times (0,T)),
	\end{equation}
as $n\rightarrow +\infty$. Since convergence in $(\ref{est20})$ also implies that $u_n\rightarrow u$ a.e. in $[0,L]\times (0,T)$, we obtain by the fact $Q(s)=|s|^2s$ is a complex continuous function that
$
	|u_n|^2u_n\rightarrow |u|^2u\ \mbox{a.e. in}\ [0,L]\times (0,T)
$
as $n\rightarrow +\infty$. An application of the dominated convergence theorem so implies
\begin{equation}\label{est21}
|u_n|^2u_n\rightarrow |u|^2u\ \mbox{strongly in}\ L^2(0,T;L_{per}^2).
\end{equation}
\indent Using convergences $(\ref{est19})$ and $(\ref{est21})$, it is possible to pass to the limit in  (\ref{approblem}) to obtain the existence of a global weak solution $u\in L^{\infty}(0,T;H_{per,m}^1)$ with $u_t\in L^{\infty}(0,T; H_{per,m}^{-1})$. It is important to mention that $|u|^2u$ in $(\ref{est21})$ is not an element of $L^2(0,T;L_{per,m}^2)$ (we guarantee only that it is an element of the entire space $L^2(0,T;L_{per}^2)$). However, this fact is irrelevant to obtain that $u$ is a weak solution as in the statement of Theorem $\ref{wellposednessSchrodinger}$. Indeed, the convergence in $(\ref{est21})$ implies 
\begin{equation}\label{est22}
	\int_0^L|u_n(\cdot,t)|^2u_n(\cdot,t)\overline\phi dx\rightarrow\int_0^L|u(\cdot,t)|^2u(\cdot,t)\overline\phi dx,\ \textrm{a.e.}\ t\in [0,T],
	\end{equation}
for all $\phi\in H_{per,m}^1$.\\ 
\indent We can
easily check that the initial conditions and the uniqueness are also satisfied. The fact that $\mathcal{E}$ and $\mathcal{F}$ are conserved quantities can be determined using the uniqueness of global solutions combined with a rudimentary calculation. This last fact is useful in order to prove, in fact that $u\in C([0,T];H_{per,m}^1)$ with $u_t\in C([0,T]; H_{per,m}^{-1})$. 
 
 \section*{Acknowledgments}

G.B. Moraes is supported by CAPES/Brazil - Finance Code 001. F. Natali is partially supported by CNPq/Brazil (grant 303907/2021-5). A. Pastor is partially supported by CNPq/Brazil (grant 303762/2019-5).

\section*{Conflict of interest} The authors declare that they have no conflict
of interest.

\end{document}